\documentclass[a4paper]{article}

\usepackage{amssymb}

\usepackage{amsfonts,amsmath,amscd,amsthm,color}

\def\keywords#1{\begin{center}{\bf Keywords}\\{#1}\end{center}} %

\newfont{\eufm}{eufm10 scaled\magstep1}

\newcommand{\cA}{\mathcal{A}}
\newcommand{\cC}{\mathcal{C}}
\newcommand{\cI}{\mathcal{I}}

\newcommand{\cD}{\mathcal{D}}
\newcommand{\cB}{\mathcal{B}}
\newcommand{\cV}{\mathcal{V}}
\newcommand{\cU}{\mathcal{U}}
\newcommand{\cM}{\mathcal{M}}
\newcommand{\cP}{\mathcal{P}}
\newcommand{\cF}{\mathcal{F}}
\newcommand{\cY}{\mathcal{Y}}
\newcommand{\cE}{\mathcal{E}}
\newcommand{\cL}{\mathcal{L}}

\newcommand{\cH}{\mathcal{H}}

\newcommand{\cS}{\mathcal{S}}
\newcommand{\cQ}{\mathcal{Q}}

\newcommand{\cO}{\mathcal{O}}
\newcommand{\cX}{\mathcal{X}}

\newcommand{\bbN}{\mathbb{N}}
\newcommand{\bbZ}{\mathbb{Z}}
\newcommand{\bbC}{\mathbb{C}}
\newcommand{\bbR}{\mathbb{R}}
\newcommand{\bbQ}{\mathbb{Q}}

\newcommand{\bbP}{\mathbb{P}}
\newcommand{\bbF}{\mathbb{F}}
\newcommand{\bbE}{\mathbb{E}}
\newcommand{\bbG}{\mathbb{G}}

\newcommand{\bbK}{\mathbb{K}}
\newcommand{\bbX}{\mathbb{X}}

\newcommand{\bbD}{\mathbb{D}}

\def\ags{{\rm ags}}
\def\para{\vspace{2mm}}

\def\dres{\partial{\rm Res}}
\def\res{{\rm Res}}

\def\dfres{\partial{\rm FRes}}
\def\cfres{\rm CFRes}

\def\ps{{\rm ps}}
\def\PS{{\rm PS}}

\def\rank{{\rm rank}}
\def\ord{{\rm ord}}
\def\lord{{\rm lord}}

\def\sat{{\rm sat}}
\def\min{{\rm min}}
\def\max{{\rm max}}

\def\gcd{{\rm gcd}}

\def\Jac{{\rm Jac}}

\def\c{{\rm c}}

\def\ldeg{{\rm ldeg}}

\def\trdeg{{\rm trdeg}}
\def\vol{{\rm vol}}

\newtheorem{thm}{Theorem}[section]
\newtheorem{lem}[thm]{Lemma}
\newtheorem{cor}[thm]{Corollary}
\newtheorem{prop}[thm]{Proposition}
\newtheorem{defi}[thm]{Definition}
\newtheorem{rem}[thm]{Remark}

\newtheorem{ex}[thm]{Example}
\newtheorem{alg}[thm]{Algorithm}

\date{} %

\begin{document}
\title{Differential elimination by differential specialization of Sylvester style matrices}

\author{Sonia L. Rueda\\
Dpto. de Matem\' atica Aplicada, E.T.S. Arquitectura.\\
Universidad Polit\' ecnica de Madrid.\\
Avda. Juan de Herrera 4, 28040-Madrid, Spain.\\
\tt{sonialuisa.rueda@upm.es}}

\maketitle

\thispagestyle{empty}

\begin{abstract}
Differential resultant formulas are defined, for a system $\cP$ of $n$ ordinary Laurent differential polynomials in $n-1$ differential variables. These are determinants of coefficient matrices of an extended system of polynomials obtained from $\cP$ through derivations and multiplications by Laurent monomials. To start, through derivations, a system $\ps(\cP)$ of $L$ polynomials in $L-1$ algebraic variables is obtained, which is non sparse in the order of derivation. This enables the use of existing formulas for the computation of algebraic resultants, of the multivariate sparse algebraic polynomials in $\ps(\cP)$, to obtain polynomials in the differential elimination ideal generated by $\cP$. The formulas obtained are multiples of the sparse differential resultant defined by Li, Yuan and Gao, and provide order and degree bounds in terms of mixed volumes in the generic case.
\end{abstract}



\section{Introduction}\label{sec-differential resultant formulas}

The algebraic treatment via symbolic computation of differential equations has gained importance in the last years \cite{WZ}, \cite{MM}, \cite{RW}. In addition, algebraic and differential elimination techniques have proven to be relevant tools in constructive, algorithmic algebra and symbolic computation \cite{Mishra}, \cite{Cox}, \cite{Cox-2}, \cite{Ritt}, \cite{Kol}.
This work establishes a bridge between the differential elimination problem for systems of ordinary differential polynomials and the use of sparse algebraic resultants. Let us consider a system of two ordinary differential polynomials in the differential indeterminates $x$ and $y$,
\begin{equation}\label{eq-deg2ord1}
\begin{array}{l}
f_1(x)=y'+y x+  x' + x x' +y x^2 +y' (x')^2,\\
f_2(x)=y+y' x+ y x' +y^2 x x' + x^2 + (x')^2.
\end{array}
\end{equation}
To eliminate the differential indeterminate $x$ (and all its derivatives), they can be seen as two differential polynomials in the differential indeterminate $x$, whose coefficients are polynomials in the differential indeterminate $y$.

\para
Differential elimination for differential polynomials can be achieved by characteristic set methods via symbolic computation algorithms \cite{H}, \cite{BLM} (implemented in the Maple package {\rm diffalg}, \cite{BH} and in the {\rm BLAD libraries} \cite{BBLAD} respectively), see also \cite{GLY}, \cite{Ritt}. These methods do not have an elementary complexity bound \cite{GKOS} and, the development of algorithms based on order and degree bounds, of the output elimination polynomials, would contribute to improve the complexity. Searching for order and degree bounds of the elimination polynomials is a problem closely related to the study of differential resultants.

\para
For a system of sparse algebraic multivariate polynomials Canny and Emiris defined in \cite{CE} a Sylvester type matrix, whose determinant is a multiple of the sparse algebraic resultant, in the generic case (defined in \cite{GKZ}). Furthermore, the sparse multivariate algebraic resultant can be represented as the quotient of two determinants, as proved in \cite{DA}. These so called Macaulay style formulas provide degree bounds and furthermore methods to predict the support of the sparse algebraic resultant. While the studies and achievements on algebraic resultants are quite numerous, the differential case is at an initial state of development. A rigorous definition of the differential resultant $\dres(\frak{P})$, of a set $\frak{P}$ of $n$ sparse generic ordinary Laurent differential polynomials in $n-1$ differential variables, has been recently presented in \cite{LYG} (and in \cite{GLY}, for the non sparse nonhomogeneous polynomial case), together with a single exponential algorithm in terms of bounds for degree and order of derivation.
A matrix representation of the sparse differential resultant does not exist even for the simplest cases and,
as noted in \cite{LYG}, having Macaulay style formulas in the differential case would improve the existing bounds for degree and order. The study of such formulas is the basis for efficient computation algorithms and, it promises to have a great contribution to the development and applicability of differential elimination techniques.

\para
The first attempt to give Macaulay style formulas for a system $\cP$ of $n$ ordinary differential polynomials, in $n-1$ differential variables, was made by G. Carr\`{a}-Ferro in \cite{CFproc}. Previous definitions of differential resultants were given for two ordinary differential operators, \cite{BT}, \cite{Ch} ( refer to \cite{CFproc}, \cite{LYG} for an extended history of these developments). The differential resultant $\cfres(\cP)$ of $\cP$ defined by Carr\`{a}-Ferro is the algebraic resultant of Macaulay \cite{Mac}, of a set of derivatives of the differential polynomials in $\cP$.
For two non sparse differential polynomials of order $1$ and degree $2$, say \eqref{eq-deg2ord1},
$\cfres(\cP)$  is the Macaulay algebraic resultant of the polynomial set $\ps=\{f_1,f'_1,f_2, f'_2\}$. This is the greatest common divisor of the determinant of all the minors of maximal order of a matrix $\cM$, whose columns are indexed by all the monomials
in $x$, $x'$ and $x''$ of degree less than or equal to $5$. The rows of $\cM$ are the coefficients of polynomials obtained by multiplying the polynomials in $\ps$ by certain monomials in $x$, $x'$ and $x''$, see \cite{CF97} and \cite{R10} for details. Observe that, even if $f_1$ and $f_2$ are nonsparse in $x$ and $x'$, the extended system $\ps$ is sparse. The polynomials in $\ps$ do not contain the monomial $(x'')^2$, thus the columns indexed by $(x'')^i$, $i=2,\ldots ,5$ are all zero and $\cfres(f_1,f_2)=0$.

Carr\`{a}-Ferro's construction is not taking into consideration the sparsity of differential polynomials and therefore it is zero in many cases, giving thus no further information. Contemporary of Carr\`{a}-Ferro's construction is the definition of the sparse algebraic resultant in \cite{GKZ} and \cite{St}. Later on, methods to compute sparse algebraic resultants were developed in \cite{CE}, \cite{DA} via Sylvester style matrices.
Therefore, an alternative natural approach to treat example \eqref{eq-deg2ord1} (using Carr\`{a}-Ferro's philosophy) would be to consider the sparse algebraic resultant formula of $\ps$ given in \cite{DA}.
A determinantal formula for $2$ generic differential polynomials of arbitrary degree and order $1$ has been recently presented in \cite{ZYG}.

The system \eqref{eq-deg2ord1} can only be sparse in the degree but if we considered the elimination of two or more differential variables, the system can be also sparse in the order of derivation of such variables. This fact motivated the works in \cite{RS}, \cite{R11} and \cite{R13} where the linear case is considered, to focus on the study of the sparsity with respect to the order of derivation, as defined in Section \ref{sec-L pols in L-1 vars}. An easy example is given by the next system of $3$ polynomials
\[\cP=\{f_1=z+x+y+y',f_2=z+t x'+y'',f_3=z+x+y'\}\]
in $3$ differential variables $x$, $y$ and $z$ w.r.t. the derivation $\partial/\partial t$.
The differential resultant of Carra'Ferro is the determinant of the next coefficient matrix, whose columns are indexed by
$y^{v}$, $x^{v}$, $\ldots$ ,$y'$, $x'$, $y$,$x$,$1$,
\[
\left[
\begin{array}{rrrcrcrcrrc}
1 & 0 & 1 & 1 & 0 & 0 & 0 & 0 & 0 & 0 & z''' \\
0 & 0 & 1 & 0 & 1 & 1 & 0 & 0 & 0 & 0 & z'' \\
0 & 0 & 0 & 0 & 1 & 0 & 1 & 1 & 0 & 0 & z' \\
0 & 0 & 0 & 0 & 0 & 0 & 1 & 0 & 1 & 1 & z \\
1 & 0 & 0 &  t & 0 & 2 & 0 & 0 & 0 & 0 & z'' \\
0 & 0 & 1 & 0 & 0 & t & 0 & 1 & 0 & 0 & z' \\
0 & 0 & 0 & 0 & 1 & 0 & 0 & t & 0 & 0 & z \\
1 & 0 & 0 & 1 & 0 & 0 & 0 & 0 & 0 & 0 & z''' \\
0 & 0 & 1 & 0 & 0 & 1 & 0 & 0 & 0 & 0 & z'' \\
0 & 0 & 0 & 0 & 1 & 0 & 0 & 1 & 0 & 0 & z' \\
0 & 0 & 0 & 0 & 0 & 0 & 1 & 0 & 0 & 1 & z
\end{array}
 \right].
\]
Thus $\cfres(\cP)=0$ and the reason is the sparsity in the order of derivation of the variable $x$ (the column indexed by $x^{v}$ is zero).

\para

In Section \ref{sec-differential resultant formulas}, differential resultant formulas are defined for a system $\cP$ of $n$ ordinary Laurent differential polynomials in $n-1$ differential variables. These are determinants of coefficient matrices of an extended system of polynomials obtained from $\cP$ through derivations and multiplications by Laurent monomials (Carr\`{a}-Ferro's construction is a particular case).
To built such formulas, in Section \ref{sec-L pols in L-1 vars} the results in \cite{R13} are extended to the nonlinear case, namely, an extended system $\ps(\cP)$ of $L$ polynomials in $L-1$ algebraic variables is obtained through the appropriate number of derivations of the elements of $\cP$, which is non sparse in the order of derivation. As explained in Section \ref{sec-L pols in L-1 vars}, this is only possible for systems $\cP$ that verify the "super essential" condition, but it is there proved that every system contains such a subsystem. For $n\geq 3$, this is a necessary step to be able to use the existing formulas for sparse algebraic polynomials in \cite{CE} applied to the system $\ps(\cP)$. An algebraic generic sparse system $\ags(\cP)$ of $L$ polynomials in $L-1$ algebraic variables associated to $\cP$ is defined, as explained in Section \ref{sec-sparse algebraic resultant}, from which a Sylvester style matrix $S(\cP)$ can be constructed using the results in \cite{CE}. The specialization of $S(\cP)$, to the differential coefficients of $\ps(\cP)$, gives a determinantal formula $\dfres(\cP)$, as explained in Section \ref{sec-differential specialization}. In \cite{CE}, $\det(S(\cP))$ is guaranteed to be nonzero (under some conditions) so, if $\dfres(\cP)=0$ then we know that it is not because of sparsity reasons, but due to the specialization final step. In Section \ref{sec-order and degree bounds} the generic case is treated. It is shown how these formulas provide order and degree bounds for the sparse differential resultant $\dres(\cP)$ of $\cP$ defined by Li, Yuan and Gao in \cite{LYG}. To achieve this goal, conditions for $\dres(\cP)$ to be a factor of the given differential resultant formulas are explored, providing degree bounds of $\dres(\cP)$ in terms of mixed volumes, under the appropriate conditions.

\section{Differential resultant formulas}\label{sec-differential resultant formulas}
Let $\cD$ be an ordinary differential domain with derivation
$\partial$. Let $U=\{u_1,\ldots ,u_{n-1}\}$ be a set of differential indeterminates over $\cD$.
By $\bbN$ we mean the natural numbers including $0$.
For $k\in\bbN$, we denote by $u_{j,k}$ the $k$th derivative of $u_j$ and for $u_{j,0}$ we simply write $u_j$.
We denote by $\{U\}$ the set of derivatives of the elements of $U$,
$\{U\}=\{\partial^k u\mid u\in U,\; k\in \bbN\}$, and  by
$\cD\{U\}$ the ring of differential polynomials in the differential
indeterminates $U$, which is a differential ring with derivation $\partial$.
For definitions in differential algebra we refer to \cite{Ritt} and \cite{Kol}.

As introduced in \cite{LYG}, the ring of Laurent differential polynomials generated by $U$ is defined to be
\[\cD\{U^{\pm}\}:=\cD[u_{j,k},u_{j,k}^{-1}\mid j=1,\ldots ,n-1, k\in\bbN]\},\]
which is a differential ring under the derivation $\partial$, (emphasize that $\cD\{U^{\pm}\}$ is just notation).
Given a subset $\cU\subset\{U\}$, we denote by $\cD[\cU]$ the ring of polynomials in the indeterminates $\cU$ and by
$\cD[\cU^{\pm}]$ the ring of Laurent polynomials in the variables $\cU$, that is
\[\cD[\cU^{\pm}]:=\cD[u,u^{-1}\mid u\in\cU].\]

Given $f\in\cD\{U^{\pm}\}$, $f=\sum_{\iota=1}^m \theta_{\iota} \omega_{\iota}$, where $\theta_{\iota}\in\cD$ and $\omega_{\iota}$ is a Laurent differential monomial in $\cD\{U^{\pm}\}$. Let us denote the {\sf differential support in $u_j$} of $f$ by
\[\frak{S}_j (f)=\{k\in\bbN \mid \partial \omega_{\iota}/ \partial u_{j,k}\neq 0 \mbox{ for some }\iota\in\{1,\ldots ,m\}\},\]
to define $\ord(f,u_j):=\max\,\frak{S}_j(f)$ and $\lord (f,u_j):=\min\,\frak{S}_j(f)$ if $\frak{S}_j (f)\neq \emptyset$, otherwise $\ord(f,u_j)=\lord(f,u_j)=-\infty$.
Thus, the order of $f$ is the maximum of $\{\ord(f,u)\mid u\in U\}$.

\para
Let $\cP:=\{f_1,\ldots ,f_n\}$ be a system of differential polynomials in $\cD\{U^{\pm}\}$. We assume that:
\begin{enumerate}
\item[($\cP$1)] The order of $f_i$ is $o_i\geq 0$, $i=1,\ldots ,n$. So that no $f_i$ belongs to $\cD$.
\item[($\cP$2)] $\cP$ contains $n$ distinct polynomials.
\item[($\cP$3)] For every $j\in\{1,\ldots ,n-1\}$ there exists $i\in\{1,\ldots ,n\}$ such that $\frak{S}_j(f_i)\neq \emptyset$.
\end{enumerate}
Let $[\cP]$ denote the differential ideal generated by $\cP$ in $\cD\{U^{\pm}\}$.
Our goal is to obtain elements of the differential elimination ideal $[\cP]\cap\cD$, using differential resultant formulas.

\para
Let us denote by $\partial\cP:=\{\partial^k f_i\mid i=1,\ldots ,n, k\in\bbN\}$ and $f_i^{[L_i]}:=\{\partial^k f_i\mid k\in [0,L_i]\cap\bbN\}$, for $L_i\in\bbN$.
For this purpose, we consider a polynomial subset $\ps$ of $\partial\cP$, a set of differential indeterminates $\cU\subset\{U\}$ and sets of Laurent differential monomials $\Omega_f$, $f\in\ps$, $\Omega$, in $\cD[\cU^{\pm}]$, verifying:
\begin{itemize}
\item[(\ps1)] $\ps=\cup_{i=1}^n f_i^{[L_i]}$, $L_i\in\bbN$,
\item[(\ps2)] $\ps\subset\cD [\cU^{\pm}]$ and $|\cU|= |\ps|-1$,
\item[(\ps3)] $\sum_{f\in\ps} |\Omega_f|=|\Omega|$ and $\cup_{f\in\ps} \Omega_f f\in \oplus_{\omega\in \Omega}\cD \omega$, ($|\Omega|$ denotes de number of elements of $\Omega$).
\end{itemize}

Under assumptions (\ps1), (\ps2) and (\ps3), we consider a total set of polynomials $\PS:=\cup_{f\in\ps} \Omega_f f$ whose elements are
\[p=\sum_{\omega\in \Omega} \theta_{p,\omega} \omega,\mbox{ with }\theta_{p,\omega}\in\cD.\] The coefficient matrix of the elements in $\PS$ as polynomials in the monomials $\Omega$, $\cM(\PS,\Omega)=(\theta_{p,\omega})$, $p\in \PS$, $\omega\in \Omega$, is an $|\Omega|\times |\Omega|$ matrix. We call
\begin{equation}\label{eq-formula}
\det (\cM(\PS,\Omega))
\end{equation}
a {\sf differential resultant formula} for $\cP$.

\begin{ex}\label{ej-CF}
A differential resultant formula was defined by Carr\`{a}-Ferro in \cite{CFproc} for a system $\cP$ of nonhomogeneous differential polynomials in $\cD\{U\}$. In \cite{CFproc}, $L_i=N-o_i$, $i=1,\ldots ,n$, $N:=\sum_{i=1}^n o_i$ and $\cU=\{u_{j,k}\mid k\in [0,N]\cap\bbN,\, j=1,\ldots ,n-1\}$. The sets of monomials $\Omega_f$, $f\in\ps$ and $\Omega$ are taken so that $\cM (\PS,\Omega)$ is the specialization of the numerator matrix of the Macaulay algebraic resultant \cite{Mac} of generic algebraic polynomials $P_f$, $f\in\ps$ of degree $\deg(P_f)=\deg(f)$ in the variables $\cU$. See \cite{CFproc} and \cite{R10} for a detailed construction and examples.
\end{ex}

If (\ps2) holds, the set
\begin{align*}
\nu(\ps):&=\{u_{j,k}\in\cU\mid k\in\frak{S}_j(f)\mbox{ for some }f\in\ps, j\in\{1,\ldots ,n-1\}\}\subseteq \cU,
\end{align*}
verifies $|\nu(\ps)| \leq |\ps|-1$. Observe that, if $|\nu(\ps)| > |\ps|-1$ we cannot guarantee the elimination of the variables in $\nu(\ps)$.

\section{A system $\ps(\cP)$ of $L$ polynomials in $L-1$ algebraic variables}\label{sec-L pols in L-1 vars}

In this section, we construct $\ps(\cP)\subset \partial\cP$ and $\cV(\cP)\subset\{U\}$ verifying (\ps1), (\ps2) and give conditions on $\cP$ so that $\cV(\cP)=\nu(\ps(\cP))$. In particular, it is precisely stated what it means for the system $\cP$ to be sparse in the order and under what conditions can this phenomenon be avoided.

\para
Let us denote $o_{i,j}:=\ord(f_i,u_j)$, which equals $-\infty$ if $\frak{S}_j(f_i)=\emptyset$ and belongs to $\bbN$ otherwise.
Let us define the {\sf order matrix} of $\cP$ by $\cO(\cP)=(o_{i,j})$.
Given $\cP_i:=\cP\backslash \{f_i\}$, $i=1,\ldots,n$, the diagonals of the matrix $\cO(\cP_i)$ are indexed by the set $\Gamma_{i}$ of all possible bijections between $\{1,\ldots ,n\}\backslash \{i\}$ and $\{1,\ldots ,n-1\}$. The {\sf Jacobi number} $J_i(\cP)$ of the matrix $\cO(\cP_i)$ (see \cite{LYG}, Section 5.2) equals
\[J_i(\cP):=\Jac(\cO(\cP_i)):=\max \left\{\sum_{j^\in\{1,\ldots ,n\}\backslash \{i\}} o_{j,\mu(j)}\mid \mu\in\Gamma_i\right\}.\]
Throughout the paper, if there is no need to specify, we will simply write $J_i$.
Observe that $J_i$ is either $-\infty$ or it belongs to $\bbN$. There exists $\mu_i\in\Gamma_i$ such that $J_i=\sum_{j\in\{1,\ldots ,n\}\backslash \{i\}} o_{j,\mu_i(j)}$ but $\mu_i$ may not be unique.

\para
The situation where $J_i\geq 0$, $i=1,\ldots ,n$ is of special interest. Let $x_{i,j}$, $i=1,\ldots ,n$, $j=1,\ldots ,n-1$ be algebraic indeterminates over $\bbQ$, the field of rational numbers.
Let $X(\cP)=(X_{i,j})$ be the $n\times (n-1)$ matrix, such that
\begin{equation}\label{eq-XP}
X_{i,j}:=\left\{\begin{array}{ll}x_{i,j},&\frak{S}_j(f_i)\neq\emptyset,\\ 0,&\frak{S}_j(f_i)=\emptyset.\end{array}\right.
\end{equation}
Let $X(\cP_i)$, $i=1,\ldots ,n$, be the submatrix of $X(\cP)$ obtained by removing its $i$th row.
It follows easily that
\begin{lem}\label{lem-JiXi}
$J_i\geq 0 \Leftrightarrow \det (X(\cP_i))\neq 0$, $i=1,\ldots ,n$.
\end{lem}
\begin{proof}

If $\det(X(\cP_i))\neq 0$ then the matrix $X(\cP_i)$ has a nonzero  diagonal. Thus, there exists $\mu\in\Gamma_i$ such that \[J_i\geq \sum_{j^\in\{1,\ldots ,n\}\backslash \{i\}} o_{j,\mu(j)}\geq 0.\]
Conversely if $J_i\geq 0$, there exists $\mu\in\Gamma_i$ such that $J_i=\sum_{j^\in\{1,\ldots ,n\}\backslash \{i\}} o_{j,\mu(j)}\geq 0$. Thus $\prod_{j^\in\{1,\ldots ,n\}\backslash \{i\}}x_{j,\mu(j)}\geq 0$ and $\det(X(\cP_i))\neq 0$.
\end{proof}

\para
The notion of super essential system of differential polynomials was introduced in \cite{R13}, for systems of linear differential polynomials and it is extended here to the nonlinear case.

\begin{defi}\label{def-superess}
The system $\cP$ is called {\sf super essential} if $\det(X(\cP_i))\neq 0$, $i=1,\ldots ,n$.
Equivalently, by Lemma \ref{lem-JiXi}, $\cP$ is super essential if $J_i\geq 0$, $i=1,\ldots ,n$.
\end{defi}

\para
For $j=1,\ldots ,n-1$ let us define integers in $\bbN$
\begin{align*}
&\gamma_j(\cP):=\min\{\lord(f_i,u_j)\mid \frak{S}_j(f_i)\neq \emptyset,\, i=1,\ldots ,n\},\\ &\gamma(\cP):=\sum_{j=1}^{n-1}\gamma_j(\cP).
\end{align*}
We write just $\gamma_j$ and $\gamma$ when there is no room for confusion.
If $J_i\geq 0$, $i=1,\ldots ,n$ then $J_i-\gamma\geq 0$ and the sets of lattice points
$[0,J_i-\gamma]\cap\bbN$ are non empty. For $i=1,\ldots ,n$, we define the set of differential polynomials
\begin{equation}\label{eq-ps}
\ps(\cP):=\cup_{i=1}^n f_i^{[J_i-\gamma]},
\end{equation}
containing $L:=\sum_{i=1}^n (J_i-\gamma+1)$ differential polynomials, whose variables belong to the set $\cV(\cP)$  of differential indeterminates
\[\cV(\cP):=\{u_{j,k}\mid k\in [\gamma_j, M_j]\cap\bbN,\, j=1, \ldots , n-1\},\]
with $M_j:=m_j-\gamma$ and $m_j:=\max\{o_{i,j}+J_i-\gamma\mid i=1,\ldots ,n\}$.
By \cite{LYG}, Lemma 5.6, if $J_i\geq 0$, $i=1,\ldots ,n$ then $\sum_{i=1}^n J_i=\sum_{j=1}^{n-1} m_j$. Thus the number of elements of $\cV(\cP)$ equals
\[\sum_{j=1}^{n-1} (M_j-\gamma_j+1)=\sum_{j=1}^{n-1} (m_j-\gamma_j-\gamma+1)=\sum_{i=1}^n J_i-n\gamma+n-1=L-1.\]
Observe that $\nu(\ps(\cP))\subseteq \cV(\cP)$ and given $j\in\{1,\ldots ,n-1\}$ we have
\begin{equation}\label{eq-sparseord}
\cup_{f\in\ps(\cP)}\frak{S}_j(f)\subseteq [\gamma_j, M_j]\cap\bbN,
\end{equation}
but we cannot guarantee that the equality holds.

\begin{defi}\label{def-sparseo}
If there exists $j$ such that (\ref{eq-sparseord}) is not an equality, we will say that the system $\cP$ is {\sf sparse in the order}.
\end{defi}

\para
It can be proved as in \cite{R13}, Section 4, that every system $\cP$ contains a super essential subsystem $\cP^{*}$ and if $\rank(X(\cP))=n-1$ then $\cP^{*}$ is unique. Namely, the system $\cP^*$ can be obtained as follows:
\begin{enumerate}
\item Consider the system $\bbP=\{p_i=c_i+\sum_{j=1}^{n-1}X_{i,j} u_j\mid l=1,\ldots, m,\}$ of algebraic polynomials in $\bbK[c_1,\ldots,c_m][U]$, $\bbK:=\bbQ(X_{i,j}\mid X_{i,j}\neq 0)$.

\item Compute a reduced Gr\"{o}bner basis $\cB=\{e_0,e_1,\ldots ,e_{m-1}\}$ of the algebraic ideal $(\bbP)$ generated by $\bbP$ in $\bbK[c_1,\ldots ,c_m][U]$, with respect to lex monomial order with $u_1>\cdots >u_{p}>c_1>\cdots>c_m$. We assume that $e_0<e_1<\cdots<e_{m-1}$. By \cite{Cox}, p. 95, Exercise 10, this can be computed through an echelon form of the coefficient matrix of the system $\bbP$.

\item Observe that at least $e_0\in\cB_0:=\cB\cap \bbK[c_1,\ldots ,c_m]$, $e_0=\sum_{l=1}^m \chi_l c_l$, $\chi_l\in\bbK$. Let $\Delta (e_0):=\{l\in\{1,\ldots ,m\}\mid \chi_l\neq 0\}$.

\item $\cP^*:=\{f_l\mid l\in \Delta (e_0)\}$.
\end{enumerate}

\begin{ex}
Let us consider the systems $\cP=\{f_1,f_2,f_3,f_4\}$ and $\cP'=\{f_1,f_2,f_3,f_5\}$ with
\[f_1=2+u_1 u_{1,1}+u_{1,2},f_2=u_1u_{1,2},f_3=u_2u_{3,1},f_4=u_{1,1}u_2,f_5=u_{1,2},\]
\[
X(\cP)=\left(\begin{array}{ccc}
x_{1,1} & 0 &0\\x_{2,1} & 0 &0\\ 0 & x_{3,2} & x_{3,3}\\ x_{4,1} & x_{4,2} & 0
\end{array}\right)
\mbox{ and }
X(\cP')=\left(\begin{array}{ccc}
x_{1,1} & 0 &0\\x_{2,1} & 0 &0\\ 0 & x_{3,2} & x_{3,3}\\ x_{4,1} & 0 & 0
\end{array}\right).
\]
$\cP$ is not super essential but since $\rank(X(P))=3$, it has a unique super essential subsystem, which is $\{f_1,f_2\}$.
$\cP'$ is not super essential and $\rank(X(P'))<3$, super essential subsystems are $\{f_1,f_2\}$, $\{f_1,f_5\}$ and $\{f_2,f_5\}$.
\end{ex}

\para
We prove next that if $\cP$ is super essential then $\cP$ is not sparse in the order (Theorem \ref{thm-columns}). For this purpose we need two preparatory lemmas.

\para
Given $j\in\{1,\ldots ,n-1\}$, the set $\cI(j):=\{i\in\{1,\ldots ,n\}\mid \frak{S}_j(f_i)\neq\emptyset\}$
is not empty, because of assumption ($\cP$3). If $\cP$ is super essential, the next lemma shows, in particular, that $|\cI(j)|\geq 2$. Given $I,I'\in\{1,\ldots ,n\}$, let $\cP_{I,I'}:=\cP\backslash\{f_I,f_{I'}\}$. Let us denote by $X(\cP_{I,I'})^j$ and $\cO(\cP_{I,I'})^j$ the submatrices of $X(\cP_{I,I'})$ and $\cO(\cP_{I,I'})$ respectively, obtained by removing their $j$th column. Observe that $\det (X(\cP_{I,I'})^j)=0$ if and only if $\Jac(\cO(\cP_{I,I'})^j)=-\infty$.

\begin{lem}\label{lem-I_Ip}
Let $\cP$ be super essential and $j\in\{1,\ldots ,n-1\}$.
\begin{enumerate}
\item Given $I\in \cI(j)$, there exists $I'\in\cI(j)\backslash\{I\}$ such that $o_{I,j}-\gamma_j\leq J_{I'}-\gamma$.

\item Given distinct $I, \overline{I}\in \cI(j)$ such that $\Jac(\cO(\cP_{I,\overline{I}})^j)=-\infty$, there exists $I'\in\cI(j)\backslash\{\overline{I}\}$ such that $o_{\overline{I},j}-\gamma_j\leq J_{I'}-\gamma$ and, if $I\neq I'$ then $o_{I',j}-\gamma_j\leq J_{I}-\gamma$.
\end{enumerate}
\end{lem}
\begin{proof}
We denote $X(\cP_{I,I'})^j$ and $\cO(\cP_{I,I'})^j$ simply by $\cX_{I,I'}$ and $\cO_{I,I'}$ in this proof. Let  $X(P)^j$ be the submatrix of $X(\cP)$ obtained by removing its $j$th column. By $r_I$ we denote the row corresponding to $f_I$ in $X(\cP)^j$ and by $\cX_{\overline{I},I,I'}$ the matrix  $X(\cP_{\overline{I},I,I'})^j$ with $\cP_{\overline{I},I,I'}:=\cP\backslash\{f_{\overline{I}},f_I,f_{I'}\}$.
\begin{enumerate}
\item By definition of super essential system, $X(\cP_{I})$ contains a nonzero diagonal. That is, there exists $I'\in \cI(j)\backslash\{I\}$ and an $(n-2)\times (n-2)$ non singular submatrix $\cX_{I,I'}$ of $X(\cP_I)$. That is $\Jac(\cO_{I,I'})\neq -\infty$ and
\[o_{I,j}-\gamma_j\leq \Jac(\cO_{I,I'})+o_{I,j}-\gamma\leq J_{I'}-\gamma.\]

\item By $1$, there exists $I'\in \cI(j)\backslash\{\overline{I}\}$ such that $\Jac(\cO_{\overline{I},I'})\neq -\infty$ and $o_{\overline{I},j}-\gamma_j\leq J_{I'}-\gamma$. If $I\neq I'$, let us assume that $\Jac(\cO_{I,I'})=-\infty$ to get a contradiction. Thus we have $\det(\cX_{\overline{I},I})=0$ and $\det(\cX_{I,I'})=0$.

    \para

    The $(n-2)\times (n-2)$ matrices $\cX_{\overline{I},I}$, $\cX_{\overline{I},I'}$ and $\cX_{I,I'}$ have $n-3$ rows in common, namely $\cX_{\overline{I},I,I'}$. Since $\det(\cX_{\overline{I},I'})\neq 0$, the rows of $\cX_{\overline{I},I,I'}$ are linearly independent. This proves that \[\rank(\cX_{\overline{I},I})=\rank(\cX_{I,I'})=\rank(\cX_{\overline{I},I,I'})=n-3.\]
    Thus row $I'$ of $\cX_{\overline{I},I}$ and row $\overline{I}$ of $\cX_{I,I'}$ are a linear combination of the rows of $\cX_{I,I',\overline{I}}$. Therefore both rows $I'$ and $\overline{I}$ of $X(\cP_I)$ can be reduced to the form $(0,\ldots ,0,\star_j,0,\ldots ,0)$.
    Thus $\det(X(\cP_I))=0$ contradicting that $\cP$ is super essential. This proves that $\Jac(\cO_{I,I'})\neq-\infty$ and
    \[o_{I',j}-\gamma_j\leq \Jac(\cO_{I,I'})+o_{I',j}-\gamma\leq J_I-\gamma.\]
\end{enumerate}
\end{proof}

\begin{lem}\label{lem-k+1}
Let $f\in\cD\{U^{\pm}\}$. If $k\in\frak{S}_j(f)$ but $k+1\notin\frak{S}_j(f)$ then $k+1\in \frak{S}_j(\partial f)$.
\end{lem}
\begin{proof}
Observe that $f=A_{-l} u_{j,k}^{-l}+\cdots +A_{-1} u_{j,k}^{-1}+A_0+A_1 u_{j,k}+\cdots +A_m u_{j,k}^m$ with $A_t\in\cD \{U^{\pm}\}$, $t=-l,\ldots ,-1,0,1,\ldots ,m$, such that $k,k+1\notin\frak{S}_j (A_t)$, $A_{-l}\neq 0$ or $A_m\neq 0$ and $l\geq 1$ or $m\geq 1$. The claim follows since
\[\partial f= \partial A_{-l}u_{j,k}^{-l}+ \cdots +\partial A_0+\cdots +\partial A_m u_{j,k}^m+ \left(\sum_{h\neq 0,h=-l}^m h A_h u_{j,k}^{h-1}\right) u_{j,k+1}.\]
\end{proof}

\begin{thm}\label{thm-columns}
If $\cP$ is super essential then
\[\cup_{f\in\ps(\cP)} \frak{S}_j (f)=[\gamma_j, M_j]\cap\bbN,\,\,\, j=1,\ldots ,n-1.\]
That is, $\ps(\cP)$ is a system of $L$ polynomials in $L-1$ algebraic indeterminates.
\end{thm}
\begin{proof}
Given $j\in\{1,\ldots ,n-1\}$, there exists $\overline{I}\in \{1,\ldots ,n\}$ such that
$m_j=o_{\overline{I},j}+J_{\overline{I}}$. Recall $M_j=m_j-\gamma$.
We can write
\[[\gamma_j, M_j]=[\gamma_j,o_{\overline{I},j}-1]\cup [o_{\overline{I},j},M_j].\]
\begin{enumerate}
\item For every $k\in [o_{\overline{I},j},M_j]\cap \bbN$,
$k-o_{\overline{I},j}\leq M_j-o_{\overline{I},j}=J_{\overline{I}}-\gamma.$
By Lemma \ref{lem-k+1}, $k\in\frak{S}_j (\partial^{k-o_{\overline{I},j}}f_{\overline{I}})$.

\item If $o_{\overline{I},j}=\gamma_j$ then the first interval is empty. If $o_{\overline{I},j}\neq \gamma_j$, there exists $\underline{I}\in\cI(j)$ such that $\ldeg(f_{\underline{I}},u_j)=\gamma_j$ and a bijection $\mu_{\underline{I}}:\{1,\ldots ,n\}\backslash\{\underline{I}\}\longrightarrow \{1,\ldots ,n-1\}$, with $I:=\mu_{\underline{I}}^{-1}(j)$ such that
\[o_{I,j}-\gamma_j\leq J_{\underline{I}}-\gamma=\sum_{l\in \{1,\ldots,n\}\backslash\{\underline{I}\}} o_{l,\mu_{\underline{I}}(l)}-\gamma.\]
If there exists $k\in ([\gamma_j,o_{I,j}]\cap\bbN)\backslash\frak{S}_j(f_{\underline{I}})$ then let us consider
\[k':=\max ([\gamma_j,k-1]\cap\bbN)\cap \frak{S}_j(f_{\underline{I}}).\]
Since
$k-k'\leq o_{I,j}-k'\leq o_{I,j}-\gamma_j\leq J_{\underline{I}}-\gamma$,
by Lemma \ref{lem-k+1}, $k\in\frak{S}_j(\partial^{k-k'}f_{\underline{I}})$.
Thus, for $\ps(f_{\underline{I}})$ as in \eqref{eq-ps}, it holds
\begin{equation}\label{eq-underI}
[\gamma_j,o_{I,j}]\cap\bbN\subseteq \cup_{f\in\ps(f_{\underline{I}})}\frak{S}_j(f).
\end{equation}
\begin{enumerate}
\item[2.1.] If $o_{I,j}\geq o_{\overline{I},j}-1$ then, by \eqref{eq-underI}
$[\gamma_j,o_{\overline{I},j}-1]\cap\bbN\subseteq \cup_{f\in\ps(f_{\underline{I}})}\frak{S}_j(f)$.
\item[2.2.] If $o_{I,j}< o_{\overline{I},j}-1$ then
$[\gamma_j,o_{\overline{I},j}-1]=[\gamma_j,o_{I,j}]\cup [o_{I,j}+1,o_{\overline{I},j}-1]$.
Consequently, if $o_{\overline{I},j}\leq o_{I,j}+J_I-\gamma$ then $[o_{I,j}+1,o_{\overline{I},j}-1] \subset [o_{I,j}+1,o_{I,j}+J_I-\gamma]$. If $o_{I,j}+J_I-\gamma<o_{\overline{I},j}$ then $\Jac(\cO(\cP_{I,\overline{I}})^j)=-\infty$ since otherwise
\[o_{\overline{I},j}\leq \gamma_j+o_{\overline{I},j}+\Jac(\cO(\cP_{I,\overline{I}})^j)-\gamma\leq o_{I,j}+J_I-\gamma.\]
By Lemma \ref{lem-I_Ip} (2), there exists $I'\in\cI(j)\backslash\{\overline{I}\}$ such that
$o_{\overline{I},j}-\gamma_j\leq J_{I'}-\gamma$ and, if $I\neq I'$ then $o_{I',j}-\gamma_j\leq J_I-\gamma$.
Note this implies
\[o_{\overline{I},j}\leq o_{I',j}+J_{I'}-\gamma\mbox{ and } o_{I',j}\leq o_{I,j}+J_I-\gamma.\]
If $o_{I',j}\leq o_{I,j}$ then
$[o_{I,j}+1,o_{\overline{I},j}-1]\subset [o_{I',j},o_{I',j}+J_{I'}-\gamma]$, otherwise $o_{I,j}< o_{I',j}$ and
\begin{align*}
[o_{I,j}+1,o_{\overline{I},j}-1]&=[o_{I,j}+1,o_{I',j}-1]\cup [o_{I',j},o_{\overline{I},j}-1]\\
&\subset [o_{I,j}+1,o_{I,j}+J_I-\gamma]\cup [o_{I',j},o_{I',j}+J_{I'}-\gamma].
\end{align*}
Thus given $k\in [o_{I,j}+1,o_{\overline{I},j}-1]\cap\bbN$,
if $k\in [o_{I,j}+1,o_{I,j}+J_I-\gamma]$ then $k-o_{I,j} \leq J_{I}-\gamma$ and, by Lemma \ref{lem-k+1}, $k\in\frak{S}_j(\partial^{k-o_{I,j}}f_{I})$. Analogously, if $k\in [o_{I',j},o_{I',j}+J_{I'}-\gamma]$ then $k\in\frak{S}_j(\partial^{k-o_{I',j}}f_{I'})$.
\end{enumerate}
\end{enumerate}
\end{proof}

\begin{ex}
Let $\cP$ be a system with $\gamma=0$,
\[\cO(\cP)=\left(\begin{array}{cc}2 & 0\\-\infty & 1\\ 2 & 0\end{array}\right),\mbox{ thus }
X(\cP)=\left(\begin{array}{cc}x_{1,1} & x_{1,2}\\0 & x_{2,3}\\ x_{3,1} & x_{3,2}\end{array}\right).\]
Then $J_1=3$, $J_2=2$ and $J_3=3$ and $\cP$ is super essential.
By Theorem \ref{thm-columns}, $\ps(\cP)$ is a system with $11$ polynomials in
$10$ algebraic variables $\cV=\{u_1,u_{1,1}\ldots ,u_{1,5}, u_2,u_{2,1},\ldots ,u_{2,3}\}$.

If we consider, for instance, the system $\ps$, with $L_1=2<J_1$, $L_2=J_2$ and $L_3=J_3$. We have $10$ polynomials in $10$ algebraic variables $\cV(\cP)$, in this case we cannot guarantee the elimination of the algebraic variables $\cV(\cP)$.
\end{ex}

\section{Sparse algebraic resultant associated to $\cP$}\label{sec-sparse algebraic resultant}

The result in Theorem \ref{thm-columns}, allows the construction of a Sylvester matrix associated to the system $\ps(\cP)$, choosing orderings on the sets $\cV(\cP)$ and $\ps(\cP)$, as it is next explained.

Through a bijection $\beta: \cV(\cP)\rightarrow \{1,\ldots ,L-1\}$ we establish an ordering of the set of variables $\cV(\cP)$.
Let $\cY=\{y_1,\ldots ,y_{L-1}\}$ be a set of $L-1$ algebraic indeterminates over $\bbQ$.
A natural bijection $\upsilon:\cY\rightarrow \cV(\cP)$ is defined by $\upsilon(y_l)=\beta^{-1}(l)$.
Given the Laurent polynomial ring $\cD[\cY^{\pm}]$, $\upsilon$
extends to a ring isomorphism
\[\upsilon:\cD [\cY^{\pm}]\rightarrow \cD[\cV(\cP)^{\pm}].\]
Monomials in $\cD[\cY^{\pm}]$ are $y^{\alpha}=y_1^{\alpha_1}\cdots y_{L-1}^{\alpha_{L-1}}$, with $\alpha=(\alpha_1,\ldots ,\alpha_{L-1})\in\bbZ^{L-1}$ and
$\upsilon(y^{\alpha})=\upsilon(y_1)^{\alpha_1}\cdots \upsilon(y_{L-1})^{\alpha_{L-1}}$.
Now given $f=\sum_{\alpha\in\bbZ^{L-1}}a_{\alpha}\upsilon(y^{\alpha})$ in $\cD[\cV(\cP)^{\pm}]$, we define the {\sf algebraic support} $\cA(f)$ of $f$ as
\[\cA(f):=\left\{\alpha\in\bbZ^{L-1}\mid a_{\alpha}\neq 0\right\}.\]

\para
A bijection $\lambda:\ps(\cP)\rightarrow \{1,\ldots ,L\}$ defines an ordering in the set $\ps(\cP)$. Let us call its inverse $\rho$.
We define the {\sf algebraic generic system associated to} $\cP$ as
\[\ags(\cP):=\left\{\sum_{\alpha\in\cA(f)} c_{\alpha}^{\lambda(f)} y^{\alpha}\mid f\in\ps(\cP) \right\},\]
where $c_{\alpha}^{\lambda(f)}$ are algebraic indeterminates over $\bbQ$.
Thus we have
\[\ags(\cP)=\left\{P_l:=\sum_{\alpha\in\cA(\rho(l))} c_{\alpha}^l y^{\alpha}\mid l=1,\ldots ,L\right\}.\]

\para
Given $l\in\{1,\ldots ,L\}$, let us consider sets of algebraic indeterminates over $\bbQ$
\[\cC_l:=\{c_{\alpha}^l\mid \alpha\in\cA(\rho(l))\}\mbox{ and }\cC:=\cup_{l=1}^L \cC_l.\]
The system of algebraic generic polynomials $\ags(\cP)$ is included in $\bbE [\cY^{\pm}]$, for $\bbE:=\bbQ(\cC)$.
Given a subsystem $\cS\subseteq\ags(\cP)$, its elements are polynomials $P=\sum_{t} a_t M_{P,t}$, with $M_{P,t}$ monomials in $\bbE[\cY^{\pm}]$. We denote by
\begin{equation}\label{eq-YcS}
\cY(\cS):=\{y\in \cY\mid \partial M_{P,t}/\partial y\neq 0 \mbox{ for some monomial of } P\in \cS\}.
\end{equation}
If $\cP$ is super essential, Theorem \ref{thm-columns} implies that $\cY(\ags(\cP))=\cY$, so $\ags(\cP)$ is a set of $L$ polynomials in $L-1$ indeterminates $\cY$.

\para

A Sylvester matrix $Syl(\ags(\cP))$ for $\ags(\cP)$ can be constructed as in \cite{CE} ( see also \cite{CP}, \cite{St} and \cite{DA}), where finite sets of monomials $\Lambda_1,\ldots ,\Lambda_L,\Lambda$ in $\bbE[\cY^{\pm}]$ are determined. Let $\langle \Lambda \rangle_{\bbE}$ denote the $\bbE$-vector space generated by $\Lambda$.
The matrix in the monomial bases of the linear map
\[\langle \Lambda_1 \rangle_{\bbE} \oplus \cdots \oplus \langle \Lambda_L \rangle_{\bbE}\rightarrow \langle \Lambda \rangle_{\bbE}:(g_1,\ldots ,g_L)\mapsto \sum g_l P_l,\]
is $Syl(\ags(\cP))$.
In \cite{CE} and  \cite{St}, it is assumed without loss of generality that
\begin{itemize}
\item[(A1)] the affine lattice generated by the Minkowski sum $\sum_{f\in\ps(\cP)} \cA(f)$ has dimension $L-1$.
\end{itemize}
This technical hypothesis is removed in \cite{CP}, thus a Sylvester matrix $Syl(\ags(\cP))$ for $\ags(\cP)$ can be constructed without any additional assumption.

\para
Let $(\ags(\cP))$ be the algebraic ideal generated by $\ags(\cP)$ in $\bbQ[\cC][\cY^{\pm}]$.
No reference was found for the next result, which is proved for the sake of completeness, although it seems natural that it should exist in the sparse algebraic resultant literature.

\begin{prop}\label{prop-D}
$\det(Syl(\ags(\cP)))\in (\ags(\cP)) \cap\bbQ[\cC]$.
\end{prop}
\begin{proof}
Let us denote $D=\det(Syl(\ags(\cP)))$ and $S=Syl(\ags(\cP))$.
Assume $\Lambda_l=\{y^{\sigma_{l,h}}\mid h=1,\ldots ,\tau_l\}$, $l=1,\ldots ,L$.
Let us choose $y^{\alpha}\in\Lambda$ and define $C_{\alpha-\sigma_{l,h}}^l$ equal to $c_{\alpha-\sigma_{l,h}}^l$ if $\alpha-\sigma_{l,h}\in\cA(\rho(l))$ and zero otherwise. Let us define the linear map
\begin{align*}
\Psi:\bbE^{\tau_1} \oplus \cdots \oplus \bbE^{\tau_L} &\rightarrow \langle \Lambda\backslash\{y^{\alpha}\} \rangle_{\bbE}\\
g=(g_{1,1},\ldots ,g_{1,\tau_1},\ldots ,g_{L,1},\ldots ,g_{L,\tau_L})&\mapsto \sum_{l=1}^{L}\sum_{h=1}^{\tau_l} g_{l,h} (y^{\sigma_{l,h}} P_l-C_{\alpha-\sigma_{l,h}}^l y^{\alpha}).
\end{align*}
The columns of $S$ are indexed by the elements of $\Lambda$.
The matrix $M(\Psi)$ of $\Psi$, in the monomial bases, is the submatrix of $S$ obtained by removing the column indexed by $y^{\alpha}$. Observe that $S$ and $M(\Psi)$ are matrices with elements in $\bbQ[\cC]$.

There exists a nonzero $g\in {\rm Ker}(\Psi)\cap \bbQ[\cC]^{\sum_{l=1}^L \tau_L}$.
We can assume w.l.o.g. that $g_{1,1}\neq 0$
There exists a nonsingular matrix $E$ such that $\det(E)=g_{1,1}$ and the first row of $E\cdot S$ has all its entries equal to zero, except for the entry in the column indexed by $y^{\alpha}$, which equals $\sum_{l=1}^{L}\sum_{h=1}^{\tau_l} g_{l,h} C_{\alpha-\sigma_{l,h}}^l$. Thus
\[g_{1,1} D=\det(E\cdot S)=\gamma\sum_{l=1}^{L}\sum_{h=1}^{\tau_l} g_{l,h} C_{\alpha-\sigma_{l,h}}^l,\]
for some $\gamma\in\bbQ[\cC]$. If we develop $D$ by the column of $S$ indexed by $y^{\alpha}$ we obtain
\[D=\sum_{l=1}^{L}\sum_{h=1}^{\tau_l} r_{l,h} C_{\alpha-\sigma_{l,h}}^l,\mbox{ with } r_{l,h}\in\bbQ[\cC],\]
which implies $g_{1,1} r_{l,h}=\gamma g_{l,h}$, $l=1,\ldots ,L$, $h=1,\ldots ,\tau_l$.
This proves
\[g_{1,1}\sum_{l=1}^{L}\sum_{h=1}^{\tau_l} r_{l,h} (y^{\sigma_{l,h}} P_l-C_{\alpha-\sigma_{l,h}}^l y^{\alpha}) =
\gamma \sum_{l=1}^{L}\sum_{h=1}^{\tau_l} g_{l,h} (y^{\sigma_{l,h}} P_l-C_{\alpha-\sigma_{l,h}}^l y^{\alpha})=0,\]
and $g_{1,1}\neq 0$ in the domain $\bbQ[\cC][\cY^{\pm}]$ implies
\[\sum_{l=1}^{L}\sum_{h=1}^{\tau_l} r_{l,h} (y^{\sigma_{l,h}} P_l-C_{\alpha-\sigma_{l,h}}^l y^{\alpha}) =0.\]
Furthermore
\[\sum_{l=1}^{L}\sum_{h=1}^{\tau_l} r_{l,h} y^{\sigma_{l,h}} P_l=y^{\alpha} \sum_{l=1}^{L}\sum_{h=1}^{\tau_l} r_{l,h} C_{\alpha-\sigma_{l,h}}^l= y^{\alpha} D.\]
Since $D\in\bbQ[\cC]$, we have
\[D=y^{-\alpha}\sum_{l=1}^{L}\sum_{h=1}^{\tau_l} r_{l,h} y^{\sigma_{l,h}} P_l\in (\ags(\cP))\cap\bbQ[\cC].\]
\end{proof}

\para

By \cite{Kol}, Chapter $0$, $\S 11$, an ideal $\cI$ in a polynomial algebra $\bbQ[\cC]$ is prime if and only if it has a generic zero $\epsilon$ in $\bbE^{|\cC|}$, for a natural field extension $\bbE$ of $\bbQ$. That is, a polynomial in $\bbQ[\cC]$ belongs to $\cI$ if and only if it vanishes at the {\sf generic zero} $\epsilon$. In the next proof, concepts as autoreduced set and pseudo remainder will be used in the algebraic case, we refer to \cite{Mishra}.

\para

Given $l\in\{1,\ldots ,L\}$, let us suppose that $\cC_l=\{c_l,c_{l,h}\mid h=1,\ldots ,h_l\}$ and $\{T_l,T_{l,h}\mid h=1,\ldots ,h_l\}=\{y^{\alpha}\mid \alpha\in \cA(\rho(l))\}$, then
\begin{equation}\label{eq-Pl}
P_l= c_l T_l +\sum_{h=1}^{h_l} c_{l,h} T_{l,h},\mbox{ with }h_l:=|\cA(\rho(l))|-1.
\end{equation}
Let $\overline{\cC}:=\cC\backslash\{c_1,\ldots ,c_L\}$ and define,
\begin{equation}\label{eq-gazero}
\epsilon_l:=-\sum_{h=1}^{h_l} c_{l,h} \frac{T_{l,h}}{T_l}\,\,\mbox{ and }\,\,\epsilon:=(\overline{\cC};\epsilon_1,\ldots ,\epsilon_L).
\end{equation}

\begin{lem}\label{lem-prime1}
The elimination ideal $(\ags(\cP))\cap\bbQ[\cC]$ is a prime ideal with $\epsilon$ as a generic zero.
\end{lem}
\begin{proof}
We only need to prove that $\epsilon$ is a generic zero of $\cI=(\ags(\cP))\cap\bbQ[\cC]$.
Given $G\in \cI$, $G=\sum_{l}\alpha_l P_l$, with $\alpha_l\in \bbQ[\cC][\cY^{\pm}]$. Since $P_l(\epsilon)=\epsilon_l T_l+\sum_{h=1}^{h_l} c_{l,h} T_{l,h}=0$ we have $G(\epsilon)=0$. Conversely, let $G\in\bbQ[\cC]$ with $G(\epsilon)=0$. For each $l\in\{1,\ldots ,L\}$, there exists a monomial $N_l$ in the variables $\cY$ such that $N_lP_l\in \bbQ[\cC][\cY]$. Furthermore, $\cA=\{N_1P_1,\ldots ,N_LP_L\}$ is an autoreduced set with $c_l$ as leaders. Let $G_0$ be the pseudo remainder of $G$ w.r.t. $\cA$, that is $M G=\sum_{l}\beta_l N_l P_l+G_0$, for some monomial $M$ in $\cY$. Observe that $G_0\in\bbQ[\overline{\cC}][\cY]$ because each $N_lP_l$ is linear in $c_l$. Hence
\[G_0=G_0(\epsilon)=M G(\epsilon)-\sum_{l}\beta_l(\epsilon) N_l P_l(\epsilon)=0\]
and
\[G=\sum_{l}\beta_l \frac{N_l}{M} P_l=\sum_l \gamma_l P_l,\,\,\, \gamma_l\in \bbQ[\cC][\cY^{\pm}].\]
Thus $G\in \cI$ and the result is proved.
\end{proof}

Let us denote the system of generic algebraic polynomials $\ags(\cP)$ by $\cS$.
The dimension of $(\cS)\cap\bbQ[\cC]$ is by definition the transcendence degree of $\bbQ(\epsilon)$ over $\bbQ$ (\cite{Kol}, Chapter $0$, $\S 11$), let us denote it by $\trdeg(\bbQ(\epsilon)/\bbQ)$.

\begin{rem}\label{rem-R}
If $\trdeg(\bbQ(\epsilon)/\bbQ)=L-1$ then  $(\cS)\cap\bbQ[\cC]$ is a prime ideal of codimension one, which implies it is a principal ideal. Namely, there exists an irreducible polynomial denoted by ${\sf R}(\cS)$ in $\bbZ[\cC]$ such that $(\cS)\cap\bbQ[\cC]=({\sf R} (\cS))$.
If $\trdeg(\bbQ(\epsilon)/\bbQ)<L-1$ we define ${\sf R}(\cS)$ to be equal to $1$.
\end{rem}

\para

A vector of coefficients for the system $\cS=\{P_1,\ldots ,P_L\}$ defines a point ${\bf c}$ of the product of complex projective spaces ${\bf P}^{h_1}\times\cdots \times {\bf P}^{h_L}$, namely
\[{\bf c}=({\bf c}_1, {\bf c}_{1,1},\ldots ,{\bf c}_{1,h_1},\ldots ,{\bf c}_L, {\bf c}_{L,1},\ldots ,{\bf c}_{L,h_L}).\]
Let us denote $P_l^{\bf c}:={\bf c}_l T_l +\sum_{h=1}^{h_l} {\bf c}_{l,h} T_{l,h}$, $\bbC^*:=\bbC\backslash\{0\}$  and define
\[Z_0:=\{{\bf c}\in {\bf P}^{h_1}\times\cdots \times {\bf P}^{h_L}\mid P_l^{\bf c}=0, l=1,\ldots ,L \mbox{ have a common solution in }(\bbC^*)^{L-1} \}.\]
By \cite{PS}, the Zariski closure $Z$ of $Z_0$ in ${\bf P}^{h_1}\times\cdots \times {\bf P}^{h_L}$ is an irreducible variety.
As defined in \cite{PS}, if the codimension of $Z$ is one then the {\sf sparse resultant} $\res(\cS)$ of the system $\cS=\ags(\cP)$ is the irreducible polynomial in $\bbZ[\cC]$ defining the hypersurface $Z$. If the codimension of $Z$ is greater than one then $\res(\cS)$ is defined to be the constant $1$.
Observe that $(\cS)\cap \bbQ[\cC]$ is included in the ideal of the variety $Z$, thus if $Z$ has codimension one then
\begin{equation}\label{eq-RSresS}
(\res(\cS))=(\cS)\cap \bbQ[\cC]=({\sf R}(\cS)).
\end{equation}
It is proved in \cite{CE}, Section 6 (see also \cite{St}) that $\det(Syl(\ags(\cP)))$ is a nonzero multiple of (the nontrivial) $\res(\cS)$, thus  $\det(Syl(\ags(\cP)))\in (\cS)\cap \bbQ[\cC]$. Note that the proof of Proposition \ref{prop-D} is needed only in the case $\res(\cS)=1$.

\para

Given $J\subseteq \{1,\ldots ,L\}$, let us consider the affine lattice $\cL_J$ generated by $\sum_{l\in J} \cA(\rho(l))$,
\[\cL_J=\left\{\sum_{l\in J} \lambda_l \alpha_l\mid \alpha_l\in\cA(\rho(l)), \lambda_l\in\bbZ, \sum_{l\in J}\lambda_l=1 \right\},\]
with $\cL:=\cL_{\{1,\ldots ,L\}}$.
Let $\rank(\cL_J)$ denote the rank of $\cL_J$. In \cite{St}, the system $\cS_J=\{P_l\mid l\in J\}$
is said to be {\sf algebraically essential} if $\rank(\cL_J)=|J|-1$ and $\rank(\cL_{J'})\geq |J'|$, for each proper subset $J'$ of $J$. The condition,
\begin{itemize}
\item[(A2)] there exists a unique algebraically essential subsystem $\cS_I$ of $\cS$.
\end{itemize}
is proved, in \cite{PS} and \cite{St}, to be a necessary and sufficient for $Z$ to have codimension one (see also \cite{GKZ}). In such case, $\res(\cS)$ coincides with $\res(\cS_I)$, considered w.r.t. the lattice $\cL_I$, and hence
\begin{equation}\label{eq-SI}
(\cS_I)\cap \bbQ[\cC_I]=({\sf R}(\cS_I))=(\res(\cS_I)),
\end{equation}
with $\cC_I=\cup_{l\in I}\cC_l$ and $(\cS_I)$ the ideal generated by $\cS_I$ in $\bbQ[\cC_I][\cY(\cS_I)^{\pm}]$, with $\cY(\cS_I)$ as in \eqref{eq-YcS}.

\para

In \cite{PS}, if $\cS$ is essential, the degree of $\res(\cS)$ in $\cC_l$, $l=1,\ldots ,L$ was proved to equal the normalized mixed volume
\begin{equation}\label{eq-MV}
MV_{-l}(\cS):=\cM(\cQ_h\mid h\in \{1,\ldots ,L\}\backslash \{l\})=\frac{\sum_{J\subset \{1,\ldots ,L\}\backslash \{l\}} (-1)^{L-|J|} \vol \left(\sum_{j\in J} \cQ_j\right)}{vol(\cQ)},
\end{equation}
where $\cQ_l$ is the convex hull of $\cA(\rho(l))$ in $\cL\otimes\bbR$, $\vol(\cQ_l)$ its $L-1$ dimensional volume, $\sum_{j\in J} \cQ_j$ is the Minkowski sum of $\cQ_j$, $j\in J$ and $\cQ$ a fundamental lattice parallelotope in $\cL$.

\para
The Sylvester matrix $Syl(\ags(\cP))$ constructed in \cite{CE} and \cite{DA} assigns a special role to $P_1$, let us denote $S_1(\cP):=Syl(\ags(\cP))$. The same construction can be done choosing $P_l$, $l=2,\ldots ,L$ as a distinguished polynomial, obtaining a matrix denoted by $S_l(\cP)$. As noted in \cite{CE}, Section 9 and \cite{DA}, Section 4.3, $S_l(\cP)$ has the minimum number of rows containing coefficients of $P_l$, its degree in the coefficients of $P_l$ coincides with the degree of $\res(\cS)$ in the coefficients of $P_l$. Furthermore, $\res(\cS)$ can be computed as the $\gcd$ in $\bbQ[\cC]$ of the determinants \begin{equation}\label{eq-Dl}
D_l(\cP):=\det(S_l(\cP)),\,\,\, l=1,\ldots ,L.
\end{equation}

\begin{ex}\label{ex-linear}
The next system $\cP=\{f_1,f_2\}$ in $\cD\{u_1\}$, is a simplified version of a predator-prey model studied in \cite{CS} that we take as a toy example,
\[
\begin{array}{l}
f_1=
a_2 x+(a_1+a_4 x) u_1+ u_{1,1} + (a_3 +a_6 x) u_1^2 + a_5 u_1^3,\\
f_2=
x'+( b_1 + b_3 x) u_1 +(b_2 +b_5 x) u_1^2 +b_4 u_1^3,
\end{array}
\]
with $a_i,b_j$ algebraic indeterminates over $\bbQ$, $\cD=\bbQ(t)[a_i,b_j]\{x\}$ and $\partial=\frac{\partial}{\partial t}$. The first attempt to eliminate the differential variable $u_1$ was done using the Maple package {\rm diffalg}, \cite{BH} (using characteristic set methods). The computation was interrupted with no answer after two hours.
We carry the example to show the elimination of $u_1$. Computations were done with Maple 15.

Since $\ps(\cP)=\{f_1,f_2,\partial f_2\}$, with
$\partial f_2=
x''+b_3 x' u_1+ (b_3 x+b_1) u_{1,1}+b_5 x' u_1^2 +(2 b_5 x +2 b_2) u_1 u_{1,1}+3 b_4 u_1^2 u_{1,1}$
and $\cV(\cP)=\{u_1,u_{1,1}\}$, we have the following associated system of algebraic generic polynomials in $y_1,y_2$
\[\ags(\cP)=\left\{\begin{array}{l} P_1=c_1+c_{11} y_1+ c_{12} y_2 +c_{13} y_1^2+ c_{14} y_1^3,\\
P_2=c_2+c_{21} y_1+ c_{22} y_1^2+ c_{23} y_1^3,\\
P_3=c_3+c_{31} y_1 + c_{32} y_2+ c_{33} y_1^2 +c_{34} y_1 y_2+ c_{35} y_1^2 y_2\end{array}\right\}.\]

Observe that $\ags(\cP)$ is algebraically essential because the linear part of the polynomials in $\ags(\cP)$, $\{c_1+c_{11} y_1+ c_{12} y_2, c_2+c_{21} y_1, c_3+c_{31} y_1 + c_{32} y_2\}$ is an algebraically essential system that verifies (A1).
Thus the algebraic resultant $\res(\cP)$ is nontrivial.
Using "toricres04", Maple 9 code for sparse (toric) resultant matrices by I.Z. Emiris, \cite{CE},
we obtain a $12\times 12$ matrix $S_1(\cP)$ whose rows contain the coefficients of the polynomials
\begin{align*}
&y_1 P_1,\, y_1y_2 P_1,\, y_1 y_2^2 P_1,\, y_1^2 P_2,\, y_1y_2 P_2,\, y_1^2y_2 P_2,\\
&y_1y_2^2P_2,\, y_1^2y_2^2 P_2,\, y_1P_3,\, y_1y_2P_3,\, y_1y_2^2P_3,\, y_1y_2^3P_3
\end{align*}
in the monomials
\[y_1,\, y_1^2,\, y_1y_2,\, y_1^2y_2,\, y_1y_2^2,\, y_1^2 y_2^2,\, y_1 y_2^3,\, y_1^2 y_2^3,\, y_1y_2^4,\, y_1^2y_2^4,\, y_1y_2^5,\, y_1^2y_2^5.\]
If the order of the input polynomials is $P_2$, $P_3$, $P_1$, we get a $13\times 13$ matrix $S_2(\cP)$ and if the order is $P_3$, $P_1$, $P_2$, the matrix $S_3(\cP)$ obtained is $11\times 11$, namely
\[
 \left[
{\begin{array}{ccccccccccc}
 {c_1} & {c}_{12} & {c}_{11} & 0
 &  {c}_{13} & 0 &  {c}_{14} & 0 & 0 &
0 & 0 \\
 {c_3} &  {c}_{32} &  {c}_{31} &
 {c}_{34} &  {c}_{33} &  {c}_{35} & 0 & 0 & 0 & 0 & 0 \\
0 & 0 &  {c_1} &  {c}_{12} & {c}_{11} & 0 &  {c}_{13} & 0 & {c}_{14}
 & 0 & 0 \\
0 & 0 &  {c_3} &  {c}_{32} &  {c}_{31} &  {c}_{34} &  {c}_{33} &
 {c}_{35} & 0 & 0 & 0 \\
0 & 0 & 0 & 0 &  {c_1} & {c}_{12} &  {
c}_{11} & 0 &  {c}_{13} & 0 &  {c}_{14} \\
0 & 0 & 0 & 0 &  {c_3} &  {c}_{32} & {
c}_{31} &  {c}_{34} &  {c}_{33}
 &  {c}_{35} & 0 \\
 {c_2} & 0 &  {c}_{21} & 0 & {c}_{22} & 0 &  {c}_{23} & 0 & 0 & 0 & 0 \\
0 &  {c_2} & 0 &  {c}_{21} & 0 & {c}_{22} & 0 &  {c}_{23} & 0 & 0 & 0 \\
0 & 0 &  {c_2} & 0 &  {c}_{21} & 0 &
{c}_{22} & 0 &  {c}_{23} & 0 & 0 \\
0 & 0 & 0 &  {c_2} & 0 &  {c}_{21} & 0 &
 {c}_{22} & 0 & {c}_{23} & 0 \\
0 & 0 & 0 & 0 &  {c_2} & 0 &  {c}_{21} & 0 &
 {c}_{22} & 0 & {c}_{23}
\end{array}}
 \right].
\]
The determinants of these matrices are
\[D_1(\cP)=-c_3 \res(\cP), D_2(\cP)=c_1^2\res(\cP)\mbox{ and }D_3(\cP)=\res(\cP).\]
\end{ex}

\section{Differential specialization}\label{sec-differential specialization}

We are ready to define differential resultant formulas for $\cP$, through the specialization of the previously defined Sylvester matrices.


\para

Given $f\in\ps(\cP)$, with $f=\sum_{\alpha\in\cA(f)} a_{\alpha}^f \upsilon (y^{\alpha})$,
let us denote by $A_f:=\{a_{\alpha}^f\mid \alpha\in\cA(f)\}$ its coefficient set and
\begin{equation}\label{eq-AP}
A(\cP):=\cup_{f\in\ps(\cP)} A_f.
\end{equation}
Given $l\in\{1,\ldots ,L\}$, such that $\rho(l)=f$, and $c_{\alpha}^l\in \cC_l$, it holds that $a_{\alpha}^{\rho(l)}\in A_f$.
Thus we can define the specialization map
\begin{align*}
\Xi:\cC\rightarrow A(\cP), \mbox{ by }\Xi(c_{\alpha}^l)=a_{\alpha}^{\rho(l)},
\end{align*}
which naturally extends to a ring epimorphism, defining $\Xi(y_l)=\upsilon(y_l)$,
\[\Xi: \bbQ[\cC][\cY^{\pm}]\rightarrow \bbQ[A(\cP)][\cV(\cP)^{\pm}].\]
$\bbQ[A(\cP)][\cV(\cP)^{\pm}]$ is included in the differential ring $\bbQ\{A(\cP)\}\{U^{\pm}\}\subseteq \cD\{U^{\pm}\}$ and obviously
\[\Xi(P_l)=\rho(l)\in\ps(\cP), l=1,\ldots ,L.\]
Let us assume that $\cP$ is supper essential to define the determinants $D_l(\cP)$, $l=1,\ldots ,L$ in \eqref{eq-Dl}.
By Proposition \ref{prop-D}, each $D_l(\cP)$ belongs to the ideal $(\ags(\cP))\cap\bbQ[\cC]$ and
\begin{equation}\label{eq-XiDl}
\Xi(D_l(\cP))\in [\cP]\cap\cD.
\end{equation}
As defined in \eqref{eq-formula}, $\Xi(D_l(\cP))$ is a differential resultant formula for $\cP$ with
\[L_i=J_i(\cP)-\gamma(\cP),\,\,\, \cU=\cV(\cP)\mbox{ and }\Omega_f=\Xi(\Lambda_{\lambda(f)}), \Omega=\Xi(\Lambda),\]
$f\in\ps(\cP)$.

\begin{ex}
To finish Example \ref{ex-linear}.
The specializations $\Xi( D_1(\cP))$, $\Xi(D_2(\cP))$ and $\Xi(D_3(\cP))$ are nonzero polynomials in the
differential elimination ideal $[\cP]\cap\cD$ (they are not included due to their size), in particular $\Xi(D_3(\cP))=\Xi(\res(\cP))$.
\end{ex}

Observe that, even for nonzero $D_l(\cP)$, $\Xi(D_l(\cP))$ could be zero, in which case the perturbation methods in \cite{DAE} could be used to obtain a nonzero differential polynomial in $[\cP]\cap\cD$. Alternatively, an algorithm to specialize step by step and obtain a factor of the specialization, which is a nonzero differential polynomial in $[\cP]\cap\cD$, is proposed next.
A similar argument is used in other specialization results as \cite{HP}, p. 168-169 and it was used in the proof of \cite{LYG}, Theorem 6.5 for a system of non sparse generic non homogeneous differential polynomials.

\para

For $i=1,\ldots ,n$, let us assume that $A_{f_i}=\{a_i,a_{i,k}\mid k=1,\ldots ,l_i\}$ and $\{\upsilon(y^{\alpha})\mid \alpha\in\cA(f_i)\}=\{M_i,M_{i,k}\mid k=1,\ldots ,l_i\}$, then
\[f_i=a_{i} M_{i}+a_{i,1} M_{i,1}+\cdots +a_{i,l_i} M_{i,l_i}.\]
In the remaining parts of this section, we consider differential indeterminates $\frak{a}_i$, $i=1,\ldots ,n$ over $\bbQ$ and the system
\[\tilde{\cP}:=\{F_i:=\frak{a}_{i} M_i+a_{i,1} M_{i,1}+\cdots +a_{i,l_i} M_{i,l_i}\mid i=1,\ldots ,n\},\]
of sparse Laurent differential polynomials in $\tilde{\cD}\{U^{\pm}\}$, with differential domain $\tilde{\cD}:=\cD\{\frak{a}_1,\ldots ,\frak{a}_n\}$.
Observe that $\ags(\tilde{\cP})=\ags(\cP)=\{P_1,\ldots ,P_L\}$, thus $D_l(\cP)=D_l(\tilde{\cP})$.
Let us assume that specialization map $\Xi:\cC\rightarrow A(\tilde{\cP})$ verifies
\begin{equation}\label{eq-Xicl}
\Xi(c_l)=\partial^k \frak{a}_i, \mbox{ if }\rho(l)=\partial^k F_i,
\end{equation}
given $l\in\{1,\ldots ,L\}$ and as in \eqref{eq-Pl}
\[P_l= c_l T_l +\sum_{h=1}^{h_l} c_{l,h} T_{l,h},\mbox{ with }h_l:=|\cA(\rho(l))|-1.\]
Thus $\Xi(\{c_l,c_{l,h}\mid h=1,\ldots ,h_l\})=A_{\rho(l)}$ and
\[A(\tilde{\cP})=\Xi(\overline{\cC})\cup \Xi({\bf c})=\Xi(\overline{\cC})\cup \{\frak{a}_1,\ldots ,\partial^{J_1-\gamma}\frak{a}_1,\ldots , \frak{a}_n,\ldots ,\partial^{J_n-\gamma}\frak{a}_n\},\]
with $\cC=\overline{\cC}\cup {\bf c}$ and ${\bf c}=\{c_1,\ldots ,c_L\}$.

\para

The idea is that, to study the specialization of $D_l(\cP)$ to the coefficients $A(\cP)$, one can first study the specialization to the coefficients  $A(\tilde{\cP})$ and then specialize $\{\frak{a}_{i}\mid i=1,\ldots ,n\}$ to $\{a_{i}\mid i=1,\ldots ,n\}$. We dedicate the rest of the section to the first part of this specialization. The results obtained will be used in Section \ref{sec-order and degree bounds} to study the case of sparse generic differential systems. The behavior of the specialization of $\{\frak{a}_{i}\mid i=1,\ldots ,n\}$ would depend on the specific domain $\cD$ to be considered.

\para

Given a nonzero differential polynomial $Q\in (\ags(\tilde{\cP}))\cap \bbQ[\cC]$, note that $\Xi(Q)\in [\tilde{\cP}]\cap \tilde{\cD}$ but we cannot guarantee that $\Xi(Q)$ is nonzero.
For a subset $\Delta$ of $\cC$, define partial specializations
\[\Xi_{\Delta}:\cC\rightarrow (\cC\backslash \Delta)\cup \Xi(\Delta),\mbox{ by }
\Xi_{\Delta}(c)=\left\{\begin{array}{ll}\Xi(c),& c\in\Delta,\\c,& c\notin\Delta,\end{array}\right.\]
which naturally extend to ring epimorphisms
\[\Xi_{\Delta}: \bbQ[\cC][\cY^{\pm}]\rightarrow \bbQ[(\cC\backslash \Delta)\cup \Xi(\Delta)][\cY^{\pm}].\]
Observe that we leave the monomials in $\bbQ[\cY^{\pm}]$ fixed for the moment and, if $\Delta=\cC$ then $\Xi_{\Delta}(\bbQ[\cC])=\bbQ[A(\tilde{\cP})]$.
Let
\[\Xi_{\cY^{\pm}}: \bbQ[A(\tilde{\cP})][\cY^{\pm}]\rightarrow \bbQ[A(\tilde{\cP})][\cV(\tilde{\cP})^{\pm}],\]
be defined by $\Xi_{\cY^{\pm}}(y_l)=\upsilon(y_l)$.

\begin{alg}\label{alg-specialize}
\begin{itemize}
\item \underline{\sf Given}  a nonzero polynomial $Q$ in $(\ags(\tilde{\cP}))\cap \bbQ[\cC]$.
\item \underline{\sf Return} a nonzero differential polynomial $H$ in $[\tilde{\cP}]\cap\tilde{\cD}$.
\end{itemize}
\begin{enumerate}
\item Let $\Delta:=\emptyset$ and $H:=Q$.

\item If $\cC\backslash\Delta=\emptyset$, return $\Xi_{\cY^{\pm}}(H)$.

\item Choose $c\in \cC\backslash\Delta$ and define $\Delta:=\Delta\cup\{c\}$.

\item If $\Xi_{\Delta}(H)\neq 0$ then $H:=\Xi_{\Delta}(H)$, go to step $2$.

\item $H=(c-\Xi(c))^s \overline{H}$, $s\in\bbN\backslash\{0\}$, set $H:=\Xi_{\Delta}(\overline{H})\neq 0$ and go to step $2$.
\end{enumerate}
\end{alg}

We prove next that the output of the previous algorithm is a nonzero differential polynomial in $[\tilde{\cP}]\cap\tilde{\cD}$.
Given $\emptyset\neq \Delta\subset \cC$, observe that $\Xi_{\Delta}(c_l)$ equals $\partial^k \frak{a}_i$, if $c_l\in\Delta$, and $c_l$ otherwise.
Let us consider ideals
\[\cI_{\Delta,{\cY}^{\pm}}:=(\Xi_{\Delta}(P_1),\ldots ,\Xi_{\Delta}(P_L))_{\bbD_{\Delta}[\cY^{\pm}]}\]
generated by $\Xi_{\Delta}(\ags(\tilde{\cP}))$ in  $\bbD_{\Delta}[\cY^{\pm}]$, with
\begin{equation*}
\bbD_{\Delta}:=\left\{
\begin{array}{l}
\bbQ[\cC]\mbox{ if }\Delta=\emptyset,\\
\bbQ[\Xi(\overline{\cC}\cap\Delta)][\overline{\cC}\backslash\Delta\cup\Xi_{\Delta}({\bf c})]\mbox{ if }\Delta\neq\emptyset.
\end{array}
\right.
\end{equation*}
Observe that $\cI_{\cC,\cY^{\pm}}$ is the ideal generated by $\Xi_{\cC}(\ags(\tilde{\cP}))$ in
\[\bbD_{\cC}[\cY^{\pm}]=\bbQ[\Xi(\overline{\cC})][\frak{a}_1,\ldots ,\partial^{J_1-\gamma}\frak{a}_1,\ldots , \frak{a}_n,\ldots ,\partial^{J_n-\gamma}\frak{a}_n][\cY^{\pm}].\]
Let $\bbK_{\Delta}:=\bbQ(\Xi(\overline{\cC}\cap\Delta))$ if $\Delta\neq\emptyset$ and $\bbK_{\emptyset}:=\bbQ$. Observe that
\begin{equation}\label{eq-kdelta}
\bbD_{\Delta}\subset \bbK_{\Delta} [\overline{\cC}\backslash\Delta\cup\Xi_{\Delta}({\bf c})].
\end{equation}

\begin{lem}\label{lem-prime}
$\cI_{\Delta,{\cY}^{\pm}}\cap\bbK_{\Delta}[\overline{\cC}\backslash\Delta\cup\Xi_{\Delta}({\bf c})]$ is a prime ideal.
\end{lem}
\begin{proof}
As in the proof of Lemma \ref{lem-prime1}, we prove that $\cI=\cI_{\Delta,{\cY}^{\pm}}\cap\bbK_{\Delta}[\overline{\cC}\backslash\Delta\cup\Xi_{\Delta}({\bf c})]$ has a generic zero. Let us define
\[\epsilon_l^{\Delta}:=-\sum_{h=1}^{h_l} \Xi_{\Delta}(c_{l,h}) \frac{T_{l,h}}{T_{l}}.\]
We can adapt the proof of  Lemma \ref{lem-prime1} to show that
$\epsilon^{\Delta}:=(\overline{\cC}\backslash\Delta;\epsilon_1^{\Delta},\ldots ,\epsilon_L^{\Delta})$
is a generic zero of $\cI$. Observe that $\Xi_{\Delta}(P_l)=\Xi_{\Delta}(c_l)T_l+\sum_h \Xi_{\Delta}(c_{l,h}) T_{l,h}$ and $\Xi_{\Delta}(c_l)$ is replaced by $\epsilon_l^{\Delta}$.
By \cite{Kol}, Chapter 0, Section 11, $\cI$ is a prime ideal.
\end{proof}

\begin{thm}\label{thm-alg}
Given a nonzero differential polynomial in $(\ags(\tilde{\cP}))\cap \bbQ[\cC]$,
the output of Algorithm \ref{alg-specialize} is a nonzero differential polynomial in $[\tilde{\cP}]\cap\tilde{\cD}$.
\end{thm}
\begin{proof}
Let $c\in\cC\backslash\Delta$, if $H\in \cI_{\Delta,\cY^{\pm}}\cap\bbD_{\Delta}$ verifies $\Xi_{\Delta}(H)=0$ then
$H=(c-\Xi(c))^s \overline{H}$ with $\overline{H}\in\bbD_{\Delta}$. By Lemma \ref{lem-prime} and \eqref{eq-kdelta}, $\overline{H}\in \cI_{\Delta,\cY^{\pm}}\cap\bbD_{\Delta}$.
For $\Delta'=\Delta\cup\{c\}$, if $H\in \cI_{\Delta,\cY^{\pm}}\cap\bbD_{\Delta}$ then \[\Xi_{\Delta'}(H)\in \cI_{\Delta',\cY^{\pm}}\cap\bbD_{\Delta'}.\]
Thus steps 4 and 5 of Algorithm \ref{alg-specialize} return polynomials in $\cI_{\Delta',\cY^{\pm}}$.
If $\Delta'=\cC$ then step 2 returns $\Xi_{\cY^{\pm}}(H)$ that belongs to $(\ps(\tilde{\cP}))$ the ideal generated by $\ps(\tilde{\cP})$ in $\bbQ[A(\tilde{\cP})][\cV(\tilde{\cP})^{\pm}]$, thus $\Xi_{\cY^{\pm}}(H)\in [\tilde{\cP}]\cap \tilde{\cD}$.
\end{proof}

Therefore, if $D_l(\tilde{P})$ is nonzero, by Proposition \ref{prop-D} it can be taken as the input of Algorithm \ref{alg-specialize} and, by Theorem \ref{thm-alg}, we obtain a nonzero differential polynomial in $[\tilde{\cP}]\cap\tilde{\cD}$.

\section{Order and degree bounds}\label{sec-order and degree bounds}

In this section, we prove that the formulas obtained are multiples of the differential resultant of a system of generic sparse differential polynomials, defined by Li, Gao and Yuan in \cite{LYG} and whose definition we include below. This fact is used to give order and degree bounds of the sparse differential resultant.

\para

Let us consider sets of differential indeterminates over $\bbQ$,
$A_i:=\{\frak{a}_{\alpha}^i\mid \alpha\in\cA(f_i)\}$, $i=1,\ldots ,n$, $A:=\cup_{i=1}^n A_i$ and
a differential domain $\frak{D}:=\bbQ\{A\}$,
to define the system $\frak{P}:=\{\bbF_1,\ldots ,\bbF_n\}$ of sparse generic differential polynomials in $\frak{D}\{U^{\pm}\}$.
The sparse generic polynomial $\bbF_i$ in $\frak{D}\{U^{\pm}\}$ with algebraic support $\cA(f_i)$ is
\[\bbF_i:=\sum_{\alpha\in\cA(f_i)} \frak{a}_{\alpha}^i \upsilon(y^{\alpha}),\]
which has order $o_i$.
The ideal generated by $\frak{P}$ in $\frak{D}\{U^{\pm}\}$ is denoted by $[\frak{P}]$.
In this section $J_i$ and $\gamma$ denote $J_i(\frak{P})$ and $\gamma(\frak{P})$ respectively. Let $A(\frak{P})$ be as in \eqref{eq-AP} and $\partial^k A_i:=\{\partial^k \frak{a}\mid \frak{a}\in A_i\}$, $k\in\bbN$. It holds
\[A\subset A(\frak{P})\mbox{ and }\bbQ[A(\frak{P})]=\bbQ[\cup_{i=1}^n A_i^{[J_i-\gamma]}],\]
with
\begin{equation}\label{eq-Aiomega}
A_i^{[\tau_i]}:=\left\{\begin{array}{l}
\cup_{i=1}^n\cup_{k=0}^{\tau_i} \partial^k A_i\mbox{ if }\tau_i\in\bbN,\\
\emptyset\mbox{ if }\tau_i=-\infty.
\end{array}\right.
\end{equation}
\para

If the differential elimination ideal $[\frak{P}]\cap\frak{D}$ has dimension $n-1$ then $[\frak{P}]\cap\frak{D}=\sat(\dres(\frak{P}))$, the saturated ideal determined by an irreducible differential polynomial $\dres(\frak{P})$, which is called the {\sf sparse differential resultant of} $\frak{P}$, \cite{LYG}, Definition 3.10. The saturated ideal of $\dres(\frak{P})$ is the set of all differential polynomials in $\frak{D}$ whose differential remainder (under any elimination ranking) w.r.t. $\dres(\frak{P})$ is zero, see \cite{Ritt} and \cite{Kol}.

\begin{lem}\label{eq-ordres1}
For every $H\in [\frak{P}]\cap\frak{D}$
\[\omega_i:=\ord(\dres(\frak{P}),A_i)\leq\ord(H,A_i),\,\,\, i=1,\ldots ,n.\]
\end{lem}
\begin{proof}
If $\ord(H,A_k)=-\infty$ then $\ord(\dres(\frak{P},A_k))=-\infty$, since otherwise $H$ cannot be reduced to zero by an elimination ranking with elementes in $A_k$ greater than the elements of $A_i$, $i\neq k$. Thus the result follows easily from $H\in \sat(\dres(\frak{P}))$.
\end{proof}


\para

For $i=1,\ldots ,n$, let us assume that $A_i=\{\frak{a}_i,\frak{a}_{i,k}\mid k=1,\ldots ,l_i\}$ and $\{\upsilon(y^{\alpha})\mid \alpha\in\cA(f_i)\}=\{M_i,M_{i,k}\mid k=1,\ldots ,l_i\}$, then
\begin{equation}\label{eq-Fi}
\bbF_i=\frak{a}_{i} M_i+\frak{a}_{i,1} M_{i,1}+\cdots +\frak{a}_{i,l_i} M_{i,l_i}.
\end{equation}
By \cite{LYG}, Definition 3.6, $\frak{P}$ is {\sf Laurent differentially essential} if for each $i\in\{1,\ldots ,n\}$ there exists $k_i\in\{1,\ldots ,l_i\}$ such that the differential transcendence degree of the set of monomials $\{M_{i,k_i}M_{i}^{-1}\mid i=1,\ldots ,n\}$ over $\bbQ$ is $n-1$.

\para

Let $\cI$ be a differential ideal in $\frak{D}=\bbQ\{A\}$.
Given a differential field extension $\cE$ of $\bbQ$, $\zeta$ in $\cE^{|A|}$ is called a {\sf generic zero} of $\cI$ if, a differential polynomial $F\in\frak{D}$ belongs to $\cI$ if and only if $F(\zeta)=0$, \cite{Ritt}, p. 27. Furthermore,
$\cI$ is prime if and only if it has a generic zero.
Given $\overline{A}:=A\backslash\{\frak{a}_1,\ldots ,\frak{a}_n\}$, the differential field extension $\bbQ\langle \overline{A},\cY^{\pm}\rangle$ of $\bbQ$ contains
\[\zeta_i:=-\sum_{k=1}^{l_i} \frak{a}_{i,k} \frac{M_{i,k}}{M_{i}}, i=1,\ldots ,n.\]
By \cite{LYG}, Corollary 3.12,
\begin{equation}\label{eq-gdzero}
\zeta:=(\overline{A};\zeta_1,\ldots ,\zeta_n)
\end{equation}
is a generic zero of the differential prime ideal $[\frak{P}]\cap\frak{D}$, which has codimension one if and only if $\frak{P}$ is Laurent differentially essential. To prove this last claim, in \cite{LYG}, Theorem 3.9, it is proved that the differential transcendence degree (see \cite{Kol}) of $\bbQ\langle \zeta \rangle$ over $\bbQ$ is $|A|-1=|\overline{A}|+n-1$, using the next result.

\begin{lem}[\cite{LYG}, proof of Theorem 3.9]
For $\cF=\bbQ\langle \overline{A}\rangle$, $\frak{P}$ is Laurent differentially essential if and only if the differential transcendence degree of $\cF\langle\zeta_1,\ldots ,\zeta_n\rangle$ over $\cF$ is $n-1$.
\end{lem}

In the remaining parts of this section, let us assume that $\frak{P}$ is Laurent differentially essential and therefore that $\dres(\frak{P})$ exists. We will use the next result about the order $\omega_i$ of $\dres(\frak{P})$ in $A_i$.


\begin{lem}[\cite{LYG}, Lemma 5.4]\label{lem-ord}
For $i=1,\ldots ,n$, if $\omega_i\geq 0$ then $\omega_i=\ord(\dres(\frak{P}),\frak{a}_{i})$ and $\omega_i=\ord(\dres(\frak{P}),\frak{a}_{i,k})$, $k=1,\ldots ,l_i$.
\end{lem}

In \cite{LYG}, order and degree bounds for $\dres(\frak{P})$ were given.
Recall that if $J_i\geq 0$, $i=1,\ldots ,n$ the system $\frak{P}$ is called super essential in Definition \ref{def-superess}. The next theorem follows from this fact and \cite{LYG}, Theorem 5.13. Emphasize that if $\frak{P}$ is super essential then $J_i-\gamma\geq 0$, $i=1,\ldots ,n$ but $\omega_i\geq 0$ is not guaranted.

\begin{thm}\label{thm-obound}
Let $\frak{P}$ be a Laurent differentially essential system. If $\frak{P}$ is super essential then
\begin{equation}\label{eq-ordres}
\omega_i=\ord(\dres(\frak{P}), A_i)\leq J_i-\gamma,\,\,\, i=1,\ldots ,n.
\end{equation}
\end{thm}

Observe that to obtain the same conclusion, in \cite{LYG}, Corollary 5.11 a much stronger condition on $\frak{P}$ than super essential was demanded (namely rank essential, see \cite{LYG}, Definition 4.20). The same conclusion is obtained later in this section by Remmarks \ref{rem-orderbound} and \ref{rem-orderbound2}.

\para

In this section, we revise the order bounds and we provide degree bounds for the sparse differential resultant of $\frak{P}$ in terms of normalized mixed volumes.
The goal of the next results is to prove Theorem \ref{thm-Qfactor} and its corollaries. They explain under which conditions the nonzero specialization of certain polynomials in the algebraic elimination ideal $(\ags(\frak{P}))\cap\bbQ[\cC]$ are multiples of the sparse differential resultant $\dres(\frak{P})$. Those specializations can then be used to give order and degree bounds of $\dres(\frak{P})$.

Given $\tau=(\tau_1,\ldots ,\tau_n)\in\bbN_{-\infty}^n$ with $\bbN_{-\infty}:=\bbN\cup\{-\infty\}$, let us define $\bbF_i^{[\tau_i]}:=\{\bbF_i,\partial \bbF_i,\ldots ,\partial^{\omega_i}\bbF_i\}$ if $\tau_i\in\bbN$ and $\bbF_i^{[\tau_i]}:=\emptyset$ if $\tau_i=-\infty$, $i=1,\ldots ,n$, $\kappa_j:=\max\{o_{i,j}+\tau_i\mid i=1,\ldots ,n\}-\gamma$, $j=1,\ldots ,n-1$ and
\begin{align*}
&\PS(\tau):=\cup_{i=1}^n \bbF_i^{[\tau_i]},\\
&\cV(\tau):=\{u_{j,k}\mid k\in [\gamma_j,\kappa_j]\cap\bbN,j=1,\ldots ,n-1\}.
\end{align*}
Observe that, if $\tau_i\leq J_i-\gamma$ then $\PS(\tau)\subseteq \ps(\frak{P})$.
Let us consider the algebraic generic system associated to $\PS(\tau)$
\[\cS_{\tau}:=\left\{\sum_{\alpha\in\cA(\bbF)} c_{\alpha}^{\lambda(\bbF)} y^{\alpha}\mid \bbF\in\PS(\tau)\right\}\subseteq \ags(\frak{P}),\]
and assume that $\cS_{\tau}=\{P_{l_1},\ldots ,P_{l_{|\PS(\tau)|}}\}$ with, as in \eqref{eq-Pl},
\begin{equation}\label{eq-Pt}
P_t:=c_t T_t+\sum_{h=1}^{h_t} c_{t,h} T_{t,h},\,\,\, t\in \Lambda(\tau):=\{l_1,\ldots ,l_{|\cS_{\tau}|}\}\subseteq\{1,\ldots ,L\}.
\end{equation}
The sets of algebraic indeterminates over $\bbQ$
\[\cC_{\tau}:=\cup_{t\in\Lambda(\tau)}\cC_{t}=\cup_{t\in\Lambda(\tau)}\{c_t,c_{t,h}\mid h=1,\ldots ,h_t\}\mbox{ and } \overline{\cC}_{\tau}:=\cC_{\tau}\backslash\{c_t\mid t\in\Lambda(\tau)\}\]
together with $\cY(\cS_{\tau}):=\{y\in\cY\mid \upsilon(y)\in\cV(\tau)\}$
describe the ring $\bbQ[\cC_{\tau}][\cY(\cS_{\tau})^{\pm}]$ that contains $\cS_{\tau}$.
As in \eqref{eq-gazero}
\[\epsilon_t:=-\sum_{h=1}^{h_t} c_{t,h}\frac{T_{t,h}}{T_t},\,\,\, t\in\Lambda(\tau).\]

We assume that the specialization map $\Xi:\bbQ[\cC][\cY^{\pm}]\rightarrow\bbQ[A(\frak{P})][\cV(\frak{P})^{\pm}]$ defined in Section \ref{sec-differential specialization} verifies \eqref{eq-Xicl}.

\begin{lem}\label{lem-res}
Let $(\cS_{\tau})$ be the ideal generated by $\cS_{\tau}$ in $\bbQ[\cC_{\tau}][\cY(\cS_{\tau})^{\pm}]$.
\begin{enumerate}
\item $(\cS_{\tau})\cap \bbQ[\cC_{\tau}]$ is a prime ideal with generic zero $\epsilon_{\tau}:=(\overline{\cC}_{\tau};\epsilon_{l_1},\ldots ,\epsilon_{l_{|\cS_{\tau}|}})$.

\item Let $\omega=(\omega_1,\ldots ,\omega_n)$, with $\omega_i=\deg(\dres(\frak{P}),A_i)$. The prime ideal $(\cS_{\omega})\cap \bbQ[\cC_{\omega}]$ has codimension one.
\end{enumerate}
Furthermore, there exists an irreducible polynomial ${\sf R}(\cS_{\omega})$ in $\bbQ[\cC_{\omega}]$ such that
$(\cS_{\omega})\cap \bbQ[\cC_{\omega}]=({\sf R}(\cS_{\omega}))$.
\end{lem}
\begin{proof}
\begin{enumerate}
\item Analogously to Lemma \ref{lem-prime1} we can prove that $\epsilon_{\tau}$ is a generic point of $(\cS_{\tau})\cap \bbQ[\cC_{\tau}]$.

\item Let us assume w.l.o.g. that $\omega_n\geq 0$ and $\Xi(P_{l_{|\cS_{\omega}|}})=\partial^{\omega_n} F_n$. By \eqref{eq-Xicl} $\Xi(c_{l_{|\cS_{\omega}|}})=\partial^{\omega_n} \frak{a}_n$. We will prove that $\epsilon_{l_1},\ldots ,\epsilon_{l_{|\cS_{\omega}|-1}}$ are algebraically independent over $\bbQ(\overline{\cC}_{\omega})$. Otherwise, there exists a nonzero polynomial $Q\in \bbQ[\overline{\cC}_{\omega}][c_{l_1},\ldots ,c_{l_{|\cS_{\omega}|-1}}]$ such that $Q(\epsilon_{l_1},\ldots ,\epsilon_{l_{|\cS_{\omega}|-1}})=0$. By definition of generic zero, $Q\in (\cS_{\omega})\cap \bbQ[\cC_{\omega}]$ and, by Theorem \ref{thm-alg}, there exists a nonzero differential polynomial $H\in [\frak{P}]\cap \frak{D}=\sat(\dres(\frak{P}))$ given by Algorithm \ref{alg-specialize}. Observe that $\partial^{\omega_n}\frak{a}_{n}$ cannot appear in $H$, by definition of $Q$. Thus $\ord(H,\frak{a}_{n})\leq \omega_n-1<\ord(\dres(\frak{P}),\frak{a}_{n})=\omega_n$, contradicting that $\dres(\frak{P})$ is the differential resultant, by Lemma \ref{lem-ord}. This proves that \[\trdeg(\bbQ(\epsilon_{\omega})/\bbQ)=|\overline{\cC}_{\omega}|+|\cS_{\omega}|-1=|\cC_{\omega}|-1\]
     and hence $(\cS_{\omega})\cap \bbQ[\cC_{\omega}]$ has codimension one. The conclusion follows as in Remark \ref{rem-R}.
\end{enumerate}
\end{proof}

The second part of the next lemma is part of the proof of \cite{LYG}, Theorem 6.5 and it is used there to bound the degree of $\dres(\frak{P})$.

\begin{lem}\label{lem-PStau}
Let $(\PS(\tau))$ be the ideal generated by $\PS(\tau)$ in  $\bbQ[\Xi(\cC_{\tau})][\cV(\tau)^{\pm}]$.
\begin{enumerate}
\item $(\PS(\tau))\cap \bbQ[\Xi(\cC_{\tau})]$ is a prime ideal with generic zero $\zeta_{\tau}$ given by \eqref{eq-zetatau}.

\item Let $\omega=(\omega_1,\ldots ,\omega_n)$, with $\omega_i=\deg(\dres(\frak{P}),A_i)$. The prime ideal $(\PS(\omega))\cap \bbQ[\Xi(\cC_{\omega})]$ has codimension one and is equal to $(\dres(\frak{P}))$.
\end{enumerate}
\end{lem}
\begin{proof}
\begin{enumerate}
\item Let $\zeta$ be as in \eqref{eq-gdzero} and $\frak{a}_i$ as in \eqref{eq-Fi}. If $\tau_i\in\bbN$, let us define the sets
\[\zeta_i^{[\tau_i]}:=\{\zeta_i,\partial\zeta_i,\ldots ,\partial^{\tau_i}\zeta_i\} \mbox{ and } \frak{a}_i^{[\tau_i]}:=\{\frak{a}_i,\ldots ,\partial^{\tau_i}\frak{a}_i\},\]
otherwise these sets are defined to be empty. $A_i^{[\tau_i]}$ was defined in \eqref{eq-Aiomega}.
Let $\frak{a}:=\cup_{i=1}^n \frak{a}^{[\tau_i]}$ and observe that
\begin{equation}\label{eq-cupdAi}
\bbQ[\Xi(\cC_{\omega})]=\bbQ[\cup_{i=1}^n A_i^{[\tau_i]}].
\end{equation}
For $\partial^k\bbF_i$ in $\PS(\omega)$ it holds
\[\partial^k\bbF_i(\zeta)=\partial^k\bbF_i(\zeta_i,\partial\zeta_i,\ldots ,\partial^k\zeta_i)=0.\]
Analogously to Lemma \ref{lem-prime1}, it can be proved that $(\PS(\omega))\cap \bbQ[\Xi(\cC_{\omega})]$ is a prime ideal with
\begin{equation}\label{eq-zetatau}
\zeta_{\tau}:=(\cup_{i=1}^n A_i^{[\tau_i]}\backslash\frak{a};\cup_{i=1}^n\zeta_i^{[\tau_i]})
\end{equation}
as a generic zero.

\item Similarly to Lemma \ref{lem-res}, (2) it follows that $|\PS(\omega)|-1$ of the elements in $\cup_{i=1}^n\zeta_i^{[\omega_i]}$ are algebraically independent and therefore this ideal has codimension one.

By Lemma \ref{lem-ord} and \eqref{eq-cupdAi}, $\dres(\frak{P})\in\bbQ[\Xi(\cC_{\omega})]$.
Since $\zeta$ in \eqref{eq-gdzero} is a generic zero of $[\frak{P}]\cap\frak{D}=\sat(\dres(\frak{P}))$, it holds $\dres(\frak{P})(\zeta)=0$, which means that replacing $\partial^k\frak{a}_i$ by $\partial^k\zeta_i$, $k=0,\ldots ,\omega_i$,  $\dres(\frak{P})$ becomes zero (for $i\in\{1,\ldots ,n\}$ with $\omega_i\geq 0$). This implies that $\dres(\frak{P})$ is an irreducible polynomial in $(\PS(\omega))\cap \bbQ[\Xi(\cC_{\omega})]$ and proves the result.
\end{enumerate}
\end{proof}

Observe that
\begin{equation}\label{eq-included}
\Xi((\cS_{\tau})\cap \bbQ[\cC_{\tau}])\subset (\PS(\tau))\cap \Xi(\bbQ[\cC_{\tau}]).
\end{equation}
The next resutl shows that the irreducible polynomials ${\sf R} (\cS_{\omega})$ and $\dres(\frak{P})$, defining respectively algebraic and differential elimination ideals of codimension one, are related through the specialization process.

\begin{prop}\label{prop-dresfactor}
Let $\frak{P}$ be a Laurent differentially essential system and $\omega=(\omega_1,\ldots ,\omega_n)$, with $\omega_i=\ord(\dres(\frak{P}),A_i)$. There exists $\frak{E}\in \bbQ[\Xi(\frak{C})]$ such that $\Xi({\sf R}(\cS_{\omega}))=\frak{E}\dres(\frak{P})$ and
\[\deg(\dres(\frak{P}), A_i^{[\omega_i]})\leq \sum_{k=0}^{\omega_i} \deg({\sf R}(\cS_{\omega}),C_{\lambda(\partial^k \bbF_i)}),\,\,\, i=1,\ldots ,n.\]
\end{prop}
\begin{proof}
By Lemmas \ref{lem-res}, \ref{lem-PStau} and \eqref{eq-included}
\[\Xi(({\sf R}(\cS_{\omega})))=\Xi((\cS_{\omega})\cap \bbQ[\cC_{\omega}])\subset (\dres(\frak{P}))=(\PS(\omega))\cap \bbQ[\Xi(\cC_{\omega})].\]
Thus $\Xi({\sf R}(\cS_{\omega}))=\frak{E}\dres(\frak{P})$, with $\frak{E}\in \bbQ[\Xi(\cC_{\omega})]$, which implies
the inequality.
\end{proof}

Since a priori we do not know the value of the orders $\omega_i$, we can use the differential resultant formulas $D_l(\frak{P})$, $l=1,\ldots ,L$ and Algorithm \ref{alg-specialize} to get new upper bounds of $\omega_i$.

\begin{rem}\label{rem-orderbound}
Assuming $\frak{P}$ to be super esential, to compute $D_l(\frak{P})$. If $\Xi(D_l(\frak{P}))\neq 0$ then, by \eqref{eq-XiDl}, it belongs to $[\frak{P}]\cap\frak{D}$. Thus, by Lemma \ref{eq-ordres1} and construction of $D_l(\frak{P})$
\[\omega_i\leq \ord(\dres(\frak{P}),A_i)\leq \ord(\Xi(D_l(\frak{P})),A_i)\leq J_i-\gamma,\,\,\, i=1,\ldots ,n,\]
which proves Theorem \ref{thm-obound} if $\Xi(D_l(\frak{P}))\neq 0$.
\end{rem}

In the remaining parts of this section, let us assume that $\frak{P}$ is super essential, to construct $D_l(\frak{P})$, $l=1,\ldots ,L$.
Given $l\in\{1,\ldots ,L\}$, let us assume that $D_l(\frak{P})\neq 0$ and $D_l(\frak{P})=Q_1\cdot\ldots\cdot Q_r$ as a product of irreducible factors in $\bbQ[\cC]$. By Proposition \ref{prop-D}, $D_l(\frak{P})\in (\ags(\frak{P}))\cap\bbQ [\cC]$, which by Lemma \ref{lem-prime1} has $\epsilon$ as a generic zero. Let
\[\cQ=\{Q\in\{Q_1,\ldots ,Q_r\}\mid Q(\epsilon)=0\},\]
which is nonempty because $(\ags(\frak{P}))\cap\bbQ [\cC]$ is prime by Lemma \ref{lem-prime1}.

\begin{lem}\label{lem-tau}
Given $Q\in\cQ$, there exists a unique $\tau^Q\in\bbN_{-\infty}^n$ such that $Q\in\bbQ[\cC_{\tau^Q}]$ and $\cC_{\tau^Q}\subset\cC_{\tau}$, for each $\tau\in\bbN_{-\infty}^n$ such that $Q\in\bbQ[\cC_{\tau}]$.
\end{lem}
\begin{proof}
Given $J=(J_1-\gamma,\ldots ,J_n-\gamma)\in\bbN^n$, $\cC=\cC_{J}$. Thus the next set is nonempty
\[\Gamma=\{\gamma=(\gamma_1,\ldots ,\gamma_n)\in\bbN_{-\infty}^n\mid Q\in\bbQ[\cC_{\gamma}]\}.\]
Then $\tau^Q=(\tau_1,\ldots ,\tau_n)$, with $\tau_i:=\min\{\gamma_i\mid \gamma\in\Gamma\}$.
\end{proof}

By Lemma \ref{lem-res}, $(\cS_{\tau^Q})\cap\bbQ[\cC_{\tau^Q}]$ is a prime ideal with generic zero $\epsilon_{\tau^Q}$. The facts that $Q\in\bbQ[\cC_{\tau^Q}]$ and $Q(\epsilon)=0$, imply that $Q(\epsilon_{\tau^Q})=Q(\epsilon)=0$. Hence $Q$ is an irreducible polynomial in $(\cS_{\tau^Q})\cap\bbQ[\cC_{\tau^Q}]$ and if this ideal has codimension one, it holds
\begin{equation}\label{eq-RStauQ}
(\cS_{\tau^Q})\cap\bbQ[\cC_{\tau^Q}]=({\sf R}(\cS_{\tau^Q}))=(Q).
\end{equation}

\begin{lem}\label{lem-ordQ}
Given $Q\in\cQ$, with $\tau^Q=(\tau_1,\ldots ,\tau_n)$ if $\Xi(Q)\neq 0$ then
\[\omega_i\leq \ord(\Xi(Q),A_i)\leq \tau_i \leq J_i-\gamma,\,\,\, i=1,\ldots ,n.\]
In particular, if $\tau_i=-\infty$ then $\omega_i=-\infty$.
\end{lem}
\begin{proof}
Observe that $Q\in\bbQ[\cC]$ implies $\tau_i\leq J_i-\gamma$. It is also easy to see that $Q\in(\cS_{\tau^Q})\cap \bbQ[\cC_{\tau^Q}]$ implies $\Xi(Q)\in (\PS(\tau^Q))\cap\bbQ[\Xi(\cC_{\tau^Q})]$ and if $\Xi(Q)\neq 0$ then $\ord(\Xi(Q),A_i)\leq \tau_i$. The first inequality follows by Lemma \ref{eq-ordres1}.
\end{proof}

We are ready now to prove the main result of this section.

\begin{thm}\label{thm-Qfactor}
Let $\frak{P}$ be a Laurent differentially essential and super essential system. Let us suppose that there exists $Q\in\cQ$ such that $(\cS_{\tau^Q})\cap\bbQ[\cC_{\tau^Q}]$ has codimension one. If $\Xi(Q)\neq 0$ then $\Xi(Q)=\frak{E} \dres(\frak{P})$, with $\frak{E}\in \frak{D}$.
\end{thm}
\begin{proof}
Let $\tau^Q=(\tau_1,\ldots ,\tau_n)$, by Lemma \ref{lem-ordQ}, $\omega_i\leq \tau_i$, $i=1,\ldots ,n$ implies
\[({\sf R}(\cS_{\omega}))=(\cS_{\omega})\cap \bbQ[\cC_{\omega}]\subseteq (\cS_{\tau^Q})\cap\bbQ[\cC_{\tau^Q}]=(Q).\]
Since $Q$ is irreducible, $Q=\alpha {\sf R}(\cS_{\omega})$, $\alpha\in \bbQ$.
By Proposition \ref{prop-dresfactor}, $\Xi({\sf R}(\cS_{\omega}))=\frak{E}_1 \dres(\frak{P})$, for some $\frak{E}_1\in\frak{D}$.
Thus
\[\Xi(Q)=\alpha \Xi({\sf R}(\cS_{\omega}))=\alpha \frak{E}_1 \dres(\frak{P}),\]
with $\frak{E}=\alpha \frak{E}_1$ in $\frak{D}$.
\end{proof}

It would be stronger to replace the assumption $\Xi(Q)\neq 0$ by $\Xi(D_l(\frak{P}))\neq 0$.

\begin{cor}\label{cor-dresFactor}
Let $\frak{P}$ be a Laurent differentially essential and super essential system. Let us suppose that there exists $Q\in\cQ$ such that $(\cS_{\tau^Q})\cap\bbQ[\cC_{\tau^Q}]$ has codimension one. If $\Xi(D_l(\frak{P}))\neq 0$ then $\Xi(D_l(\frak{P}))=\frak{E} \dres(\frak{P})$, with $\frak{E}\in \frak{D}$.
\end{cor}
\begin{proof}
Since $D_l(\frak{P})=Q'\cdot Q$, with $Q'\in\bbQ[\cC]$, by Theorem \ref{thm-Qfactor}, we get
\[\Xi(D_l(\frak{P}))=\Xi(Q')\Xi(Q)=\Xi(Q')\frak{E}_1 \dres(\frak{P}),\]
with $\frak{E}=\Xi(Q')\frak{E}_1$ in $\frak{D}$.
\end{proof}

As in Section \ref{sec-sparse algebraic resultant}, let $\res(\cS)$ be the sparse algebraic resultant of the sparse algebraic generic system $\cS=\ags(\frak{P})$. Even stronger than the assumption of the previous corollary is to assume that $\res(\cS)$ is nontrivial and $\Xi(\res(\cS))\neq 0$.

\begin{cor}\label{cor-XresFactor}
Let $\frak{P}$ be a Laurent differentially essential and super essential system. Let us suppose that $\res(\cS)$ is nontrivial. If $\Xi(\res(\cS))\neq 0$ then $\Xi(\res(\cS))=\frak{E} \dres(\frak{P})$, with $\frak{E}\in \frak{D}$.
\end{cor}
\begin{proof}
As explained in Section \ref{sec-sparse algebraic resultant}, $Q=\res(\cS)\in\cQ$ and
\[\res(\cS)\in (\cS_{\tau^Q})\cap\bbQ[\cC_{\tau^Q}]\subset (\cS)\cap\bbQ[\cS]=(\res(\cS)).\]
Thus $(\cS_{\tau^Q})\cap\bbQ[\cC_{\tau^Q}]$ has codimension one and by Theorem \ref{thm-Qfactor} the result follows.
\end{proof}

It is natural to wonder what are the conditions on $\frak{P}$ to guarantee $\Xi(D_l(\frak{P}))\neq 0$ (or $\Xi(Q)\neq 0$ for some $Q\in\cQ$) but these are not even available so far in the linear case, see \cite{R13}.
This question is left as a future research direction.

\begin{ex}
Let us consider the generic sparse differential system
$\frak{P}=\{\bbF_1=\frak{a}_1+\frak{a}_{11} u_1 u_2,\bbF_2=\frak{a}_2+\frak{a}_{21} u_1 u_{22}, \bbF_3=\frak{a}_3+\frak{a}_{31} u_{21}\}$, which is easily Laurent differentially essential and super essential. The modified Jacobi numbers are $J_1-\gamma=1,J_2-\gamma=1,J_3-\gamma=2$ and
\begin{align*}
\ps(\frak{P})=\{&\partial\bbF_1=\partial \frak{a}_1+ \partial \frak{a}_{11} u_1 u_2+\frak{a}_{11}u_{11} u_2+\frak{a}_{11} u_1 u_{21}, \bbF_1=\frak{a}_1+\frak{a}_{11} u_1 u_2,\\
 &\partial\bbF_2=\partial \frak{a}_2+\partial \frak{a}_{21} u_1 u_{22}+ \frak{a}_{21} u_{11} u_{22}+ \frak{a}_{21} u_1 u_{23}, \bbF_2=\frak{a}_2+\frak{a}_{21} u_1 u_{22},\\
 &\partial^2\bbF_3=\partial^2 \frak{a}_3+ \partial^2 \frak{a}_{31} u_{21}+2 \partial \frak{a}_{31} u_{22}+\frak{a}_{31} u_{23}, \partial\bbF_3=\partial \frak{a}_3+ \partial \frak{a}_{31} u_{21}+ \frak{a}_{31} u_{22}, \bbF_3=\frak{a}_3+\frak{a}_{31} u_{21}\}.
\end{align*}
The generic algebraic system associated to $\frak{P}$ is
\begin{align*}
\cS=\ags (\frak{P})=\{&P_1=c_{10}+c_{11} y_2 y_1+c_{12} y_2 y_3+c_{13} y_4 y_1, P_2=c_{20}+c_{21} y_2 y_1,\\
        &P_3=c_{30}+c_{31} y_5 y_1+c_{32} y_5 y_3+c_{33} y_6 y_1, P_4=c_{40}+c_{41} y_5 y_1,\\
        &P_5=c_{50}+c_{51} y_4+c_{52} y_5+c_{53} y_6, P_6=c_{60}+c_{61} y_4+c_{62} y_5, P_7=c_{70}+c_{71} y_4\}.
\end{align*}
We compute $D_1(\frak{P})$ (using "toricres04", \cite{CE}, with Maple 15) which has the next irreducible factors
\begin{align*}
&Q_1=c_{62}, Q_2=c_{40}, Q_3=-c_{70} c_{62} c_{51} + c_{70} c_{61} c_{52} - c_{60} c_{71} c_{52} + c_{62} c_{50} c_{71}\\
&Q_4=-c_{61} c_{70} + c_{71} c_{60}, Q_5=c_{70}, Q_6=c_{20} c_{40} c_{12} c_{41} c_{33} c_{71}^2 c_{62} c_{50}\\
&- c_{62} c_{40}  c_{70} c_{53} c_{32} c_{13} c_{21}- c_{62} c_{40} c_{70} c_{20} c_{51} c_{12} c_{41} c_{33} - c_{71} c_{40} c_{20} c_{12} c_{41} c_{60} c_{33} c_{52}\\
&+ c_{71} c_{60} c_{40} c_{53} c_{10} c_{32} c_{41} c_{21}+ c_{71} c_{40} c_{20} c_{12} c_{41} c_{60} c_{53} c_{31}
- c_{71} c_{20} c_{60} c_{30} c_{12} c_{41}^2  c_{53}\\
&- c_{71} c_{40} c_{20} c_{32} c_{41} c_{60} c_{53} c_{11}+ c_{40} c_{70} c_{20} c_{61} c_{52} c_{12} c_{41} c_{33}
- c_{40} c_{70} c_{53} c_{10} c_{61} c_{32} c_{41} c_{21}\\
&- c_{40} c_{70} c_{20} c_{12} c_{41} c_{53} c_{31} c_{61}+ c_{70} c_{20} c_{30} c_{61} c_{12} c_{41}^2  c_{53}
+ c_{40} c_{70} c_{20} c_{32} c_{41} c_{53} c_{11} c_{61}.
\end{align*}
Only $Q_6(\epsilon)=0$, thus $\cQ=\{Q_6\}$ and its specialization $\Xi(Q_6)\neq 0$,
\begin{align*}
\Xi(Q_6)=&-\frak{a}_{21}(-\frak{a}_{21}\frak{a}_1\frak{a}_2 \frak{a}_{11} \frak{a}_{31}^2\partial^2\frak{a}_3
+2 \frak{a}_{21} \frak{a}_{31} \frak{a}_2\frak{a}_1\frak{a}_{11}\partial\frak{a}_3\partial\frak{a}_{31}
-\frak{a}_{21}\frak{a}_{31}^2\partial\frak{a}_3\frak{a}_2\partial\frak{a}_1\frak{a}_{11}\\
&+\frak{a}_{21}\frak{a}_{31}\frak{a}_2\frak{a}_3\frak{a}_1\partial^2\frak{a}_{31}\frak{a}_{11}+
\frak{a}_{21}\frak{a}_{31}^2\frak{a}_2\frak{a}_1\partial\frak{a}_3\frak{a}_1-
2\frak{a}_{21}\frak{a}_2\frak{a}_3\frak{a}_1\partial\frak{a}_{31}^2\frak{a}_{11}\\
&+\frak{a}_{21}\frak{a}_2\frak{a}_3\frak{a}_{31}\partial\frak{a}_1\partial\frak{a}_{31}\frak{a}_{11}
+\frak{a}_{21}\frak{a}_{31}^2\frak{a}_1\partial\frak{a}_3\partial\frak{a}_2\frak{a}_{11}
-\frak{a}_{21}\frak{a}_3\frak{a}_1\partial\frak{a}_2\partial\frak{a}_{31}\frak{a}_{11}\frak{a}_{31}\\
&-\frak{a}_{21}\frak{a}_2\frak{a}_3\frak{a}_1\frak{a}_{31}\partial\frak{a}_{11}\partial\frak{a}_{31}
+\frak{a}_{31}^2\frak{a}_2^2\frak{a}_3\frak{a}_{11}^2
+\frak{a}_2\frak{a}_3\frak{a}_1\frak{a}_{11}\frak{a}_{31}\partial\frak{a}_{21}\partial\frak{a}_{31}
-\frak{a}_{31}^2\frak{a}_2\frak{a}_1\frak{a}_{11}\partial\frak{a}_3\partial\frak{a}_{21}).
\end{align*}
Observe that
$\epsilon_2,\ldots ,\epsilon_7$ are algebraically independent, since we can choose monomials
$y_1y_2$, $y_1y_5$, $y_3y_5$, $y_4$, $y_5$, $y_6$
respectively in each of them that are algebraically independent. Therefore $(\cS)\cap\bbQ[\cC]=(Q_6)$ and $\Xi(Q_6)=\frak{E}\dres(\frak{P})$. We can see that $\Xi(Q_6)=-\frak{a}_{21} H$, with $H(\zeta)=0$. Thus $H=\dres(\frak{P})$, which illustrates Theorem \ref{thm-Qfactor} and in particular Corollary \ref{cor-dresFactor}.
With a bit more work, we can prove that $\cS$ is algebraically essential so $Q_6$ is in fact the sparse algebraic resultant $\res(\cS)$. Therefore this example illustrates Corollary \ref{cor-XresFactor} as well.
\end{ex}

Another question for future investigation is, to give conditions on $\frak{P}$ so that $D_l(\frak{P})\neq 0$. Thus far we assume $D_l(\frak{P})\neq 0$ and remove the assumption $\Xi(D_l(\frak{P}))\neq 0$ (or $\Xi(Q)\neq 0$), making use of Algorithm \ref{alg-specialize}.
If $\Xi(D_l(\frak{P}))=0$, using Algorithm \ref{alg-specialize} with $D_l(\frak{P})$ as an input, by Theorem \ref{thm-alg}, we obtain $H_0$ in $[\frak{P}]\cap\frak{D}$.

\begin{rem}\label{rem-orderbound2}
Assuming $\frak{P}$ is a Laurent differentially essential and super essential system with $D_l(\frak{P})\neq 0$,
by Lemma \ref{eq-ordres1} and construction of $D_l(\frak{P})$ and $H_0$
\[\omega_i\leq \ord(H_0,A_i)\leq J_i-\gamma,\,\,\, i=1,\ldots ,n.\]
This proves Theorem \ref{thm-obound} under the assumption $D_l(\frak{P})\neq 0$.
\end{rem}

Let us assume that $H_0=H_1\cdot\ldots\cdot H_s$ as a product of irreducible factors.
Since $[\frak{P}]\cap\frak{D}$ is prime with $\zeta$ as a generic zero, the next set is nonempty
\[\cH:=\{H\in\{H_1,\ldots ,H_s\}\mid H(\zeta)=0\}.\]

\begin{lem}\label{lem-ordH0}
Given $H\in\cH$ and $\sigma=(\sigma_1,\ldots ,\sigma_n)$, with $\sigma_i:=\ord(H,A_i)$, it holds
\[\omega_i\leq \sigma_i\leq J_i-\gamma,\,\,\, i=1,\ldots ,n.\]
\end{lem}
\begin{proof}
Since $H(\zeta)=0$ and $\zeta$ is a generic zero of $[\frak{P}]\cap\frak{D}$ then $H\in [\frak{P}]\cap\frak{D}$. By Lemma \ref{eq-ordres1} and construction of $D_l(\frak{P})$ and $H$, the result follows.
\end{proof}

Given $H\in\cH$ and $\sigma=(\sigma_1,\ldots ,\sigma_n)$, with $\sigma_i:=\ord(H,A_i)$, by Lemma \ref{lem-PStau} $(\PS(\sigma))\cap \bbQ[\Xi(\cC_{\sigma})]$ is a prime ideal with $\zeta_{\sigma}$ as a generic zero.
We have $H\in\bbQ[\Xi(\cC_{\sigma})]$ and $H(\zeta)=0$, thus $H(\zeta_{\sigma})=0$ and $H\in (\PS(\sigma))\cap \bbQ[\Xi(\cC_{\sigma})]$. If $\cI(H):=(\PS(\sigma))\cap \bbQ[\Xi(\cC_{\sigma})]$ has codimension one then $\cI(H)=(H)$ and, by Lemma \ref{lem-ordH0},
\begin{equation}\label{eq-IH0}
(\dres(\frak{P}))=(\PS(\omega))\cap\bbQ[\Xi(\cC_{\omega})]\subseteq (H).
\end{equation}
Hence $H=\alpha \dres(\frak{P})$, $\alpha\in\bbQ$ because $H$ is irreducible. The previous construction proves the next result.

\begin{thm}
Let $\frak{P}$ be Laurent differentially essential and super essential system. Let $H_0$ be the output of Algorithm \ref{alg-specialize} with $D_l(\frak{P})\neq 0$ as an input.
If there exists $H\in\cH$ such that $\cI(H)$ has codimension one then $H_0=\frak{E}\dres(\frak{P})$, with $\frak{E}\in\frak{D}$.
\end{thm}


\para

To finish, we give degree bounds of $\dres(\frak{P})$ in terms of normalized mixed volumes.
As mentioned in Section \ref{sec-sparse algebraic resultant}, if $\res(\cS)$ is nontrivial, for $\cS=\ags(\frak{P})$ then $D_l(\frak{P})\neq 0$ has $\res(\cS)$ as an irreducible factor.
Furthermore, if $\cS$ is algebraically essential (and hence $\res(\cS)$ non trivial)
\begin{equation}\label{eq-boundMV}
\deg(\res(\cS),C_{\lambda(f)})=MV_{-\lambda(f)}(\cS),\,\,\, f\in\ps(\frak{P}),
\end{equation}
as in \eqref{eq-MV}.
\begin{thm}
Let $\frak{P}$ be a Laurent differentially essential and super essential system such that $\cS=\ags(\frak{P})$ is algebraically essential. If $\Xi(\res(\cS))\neq 0$ then
\[\deg(\dres(\frak{P}),A_i^{[\tau_i]})\leq \deg(\Xi(\res(\cS)), A_i^{[\tau_i]})\leq \sum_{k=0}^{\tau_i} \deg(\res(\cS),C_{\lambda(\partial^k f_i)})=\sum_{k=0}^{\tau_i} MV_{-\lambda(\partial^k f_i)}(\cS),\]
with $\tau^{\res(\cS)}=(\tau_1,\ldots ,\tau_n)$ given by Lemma \ref{lem-tau}.
\end{thm}
\begin{proof}
By Corollary \ref{cor-XresFactor}, $\Xi(\res(\cS))=\frak{E}\dres(\frak{P})$. The result follows from Lemma \ref{lem-ordQ} and \eqref{eq-boundMV}.
\end{proof}

\begin{ex}
If $\frak{P}=\{\bbG_1,\ldots ,\bbG_n\}$ is a non sparse system, with $\bbG_i$ a nonhomogeneous generic polynomial of order $o_i$ and degree $d_i$, it was proven in \cite{LYG}, Theorem 6.18 that $\cS=\ags(\frak{P})$ is algebraically essential and degree bounds for $\dres(\frak{P})$ in terms of mixed volumes are given in this case.
\end{ex}

If $\Xi(\res(\cS))=0$, we can use $\res(\cS)$ as an input of Algorithm \ref{alg-specialize}, that returns a nonzero polynomial $H_0\in [\frak{P}]\cap\frak{D}$. Assuming $H_0=H_1\cdot\ldots \cdot H_s$ as a product of irreducible factors, the set $\cH(\res(\cS))=\{H\in\{H_1,\ldots ,H_s\}\mid H(\zeta)=0\}$ is nonempty. Given $H\in\cH(\res(\cS))$, we have $\sigma_i=\ord(H,A_i)\leq \tau_i$, and by construction of $H_0$
\begin{equation}\label{eq-MVbound}
\deg(H, A_i^{[\sigma_i]})\leq \deg(H_0, A_i^{[\tau_i]})\leq \sum_{k=0}^{\tau_i} \deg(\res(\cS),C_{\lambda(\partial^k f_i)}).
\end{equation}
Furthermore, if $\cI(H)$ has codimension one then $\dres(\frak{P})=\alpha H$, $\alpha\in\bbQ$. If $\cS$ is algebraically essential then $\res(\cS)$ is nontrivial and the next result follows from \eqref{eq-MVbound} and \eqref{eq-boundMV}.

\begin{thm}
Let $\frak{P}$ be a Laurent differentially essential and super essential system such that $\cS=\ags(\frak{P})$ is algebraically essential. If there exists $H\in\cH(\dres(\frak{P}))$ such that $\cI(H)$ has codimension one then
\[\deg(\dres(\frak{P}), A_i^{[\sigma_i]})= \deg(H, A_i^{[\sigma_i]})\leq \sum_{k=0}^{\tau_i} \deg(\res(\cS), C_{\lambda(\partial^k f_i)})=\sum_{k=0}^{\tau_i} MV_{-\lambda(\partial^k f_i)}(\cS),\]
with $\tau^{\res(\cS)}=(\tau_1,\ldots ,\tau_n)$ given by Lemma \ref{lem-tau} and $\sigma_i=\ord(H,A_i)$.
\end{thm}

\section{Conclusions}

Given a system $\cP$ of $n$ Laurent sparse differential polynomials in $n-1$ differential variables $U$ ($\cP\in\cD\{U\}$), a method has been designed to compute differential resultant formulas, which provide differential polynomials where the variables $U$ have been eliminated, elements of the differential elimination ideal $[\cP]\cap\cD$. The steps of this method are:
\begin{enumerate}
\item Through derivation obtain an extended system $\ps(\cP)$ of $\cP$ with $L$ polynomials in $L-1$ variables.
\item Compute determinants of Sylvester matrices $D_l(\cP)$, $l=1,\ldots ,L$ of an sparse algebraic generic system $\cS$ associated to $\cP$.
\item The specialization $\Xi(D_l(\cP))$ of $D_l(\cP)$ to the coefficients of $\ps(\cP)$ is a differential resultant formula for $\cP$ and it belongs to the differential elimination ideal $[\cP]\cap\cD$.
\end{enumerate}
For a generic system $\frak{P}$, if $\Xi(D_l(\frak{P}))$ is nonzero then it was shown that, under appropriate conditions on $\cS$, it is a multiple of the sparse differential resultant $\dres(\frak{P})$ defined by Li, Yuan and Gao in \cite{LYG}. If $\Xi(D_l(\frak{P}))=0$, for $D_l(\frak{P})\neq 0$ an algorithm is given to still obtain an element $H$ of the differential elimination ideal, which is proved to be a multiple of the sparse differential resultant, in the appropriate situation.
If the sparse algebraic resultant $\res(\cS)$ is nontrivial, its degree can be computed in terms of normalized mixed volumes, we use those to give bounds of the degree of $\dres(\frak{P})$.

\para

It would be interesting to study the appropriate conditions on $\frak{P}$ that guarantee to have nonzero determinats $D_l(\frak{P})$, $l=1,\ldots ,L$, or furthermore $\Xi(D_l(\frak{P}))\neq 0$.
If $\Xi(D_l(\frak{P}))\neq 0$, it still remains to give close formulas to describe $\dres(\frak{P})$ as a quotient of two determinants, as it was done by D'Andrea in the algebraic case, \cite{DA}.
If $\Xi(D_l(\frak{P}))= 0$, one would need to have more control on the method to obtain the multiple $H$ of $\dres(\frak{P})$, to give closed formulas to compute $H$ and the extraneous factor.

\para

\noindent {\sf Acknowledgments:} The author kindly thanks Alfonso C. Casal and Scott MacCallum for interesting discussions on these topics, and specially Carlos D'Andrea for helpful comments on this manuscript. This work was developed, and partially supported, under the research project
MTM2011-25816-C02-01. The author belongs to the Research Group ASYNACS (Ref. {\sf CCEE2011/R34}).

\end{document}